\documentclass{amsart}

\usepackage{amsmath, amsthm, amssymb, xypic, setspace, epic}
\usepackage{hyperref}
\usepackage[usenames, dvipsnames]{color}

\hypersetup{
    colorlinks = true,
    linkcolor = RoyalBlue,
    citecolor = Cerulean
}

\numberwithin{equation}{section}

\theoremstyle{definition}
\newtheorem{definition}{Definition}[section]

\newtheorem{example}[definition]{Example}
\newtheorem{remark}[definition]{Remark}

\theoremstyle{theorem}
\newtheorem{proposition}[definition]{Proposition}
\newtheorem{theorem}[definition]{Theorem}
\newtheorem{corollary}[definition]{Corollary}
\newtheorem{lemma}[definition]{Lemma}

\newcommand{\norm}[1]{\left\lVert#1\right\rVert}
\newcommand{\ddb}{\partial\bar{\partial}}

\title{On Yau's Theorem for Asymptotically Conical Orbifolds}
\author{Mitchell Faulk}
\address{Department of Mathematics, Columbia University, New York, NY}
\email{faulk@math.columbia.edu}

\begin{document}

\maketitle

\begin{abstract}
A notion of asymptotically conical K\"ahler orbifold is introduced, and, following existence results of \cite{ch13} in the setting of asymptotically conical manifolds, it is shown that a certain complex Monge-Amp\'ere equation admits a rapidly decaying solution (which is unique for certain intervals of decay rates), allowing one to construct K\"ahler metrics with prescribed Ricci forms. In particular, if the orbifold has trivial canonical bundle, then Ricci-flat metrics can be constructed, provided certain additional hypotheses are met. This implies for example that orbifold crepant partial resolutions of varieties associated to Calabi-Yau cones admit a one-parameter family of Calabi-Yau metrics in each K\"ahler class that contains positive $(1,1)$-forms. 
\end{abstract}

\tableofcontents

\section{Introduction}

The problem of finding canonical K\"ahler metrics on K\"ahler manifolds has been a fruitful direction of research in complex geometry. It is natural to consider the same problem for orbifolds, not just for the extra degree of generality, but also because the study of such problems on orbifolds can be helpful to the same problem on manifolds (see for example \cite{gk} and the gluing construction of \cite{rs}). Moreover, in the program of mirror symmetry, it is already well-known that orbifolds can arise as the mirrors of manifolds, and hence the study of their complex geometry merits at least as much concern as that of manifolds.   

This paper concerns itself with the problem of finding canonical metrics on non-compact K\"ahler orbifolds $(\mathcal{X}, \omega)$ that are ``asymptotically conical'' in the sense that away from some compact set the orbifold is isomorphic to a K\"ahler cone in such a way that the difference between the two K\"ahler structures decays rapidly. In particular, just as in the setting of compact setting \cite{faulk}, one can search in a fixed K\"ahler class for a K\"ahler metric on $\mathcal{X}$ with a prescribed Ricci form, and it can be shown that the existence of such a metric is implied by the existence of a smooth solution $\varphi$ to the following complex Monge-Amp\'ere equation 
\begin{align}\label{eqn:MA}
\begin{cases}
(\omega + \sqrt{-1} \ddb \varphi)^n = e^F \omega^n \\
\omega + \sqrt{-1} \ddb \varphi \; \text{is a positive form}
\end{cases}
\end{align}
where, here, $n$ is the complex dimension of $\mathcal{X}$ and $F$ is a smooth fixed function on $\mathcal{X}$, satisfying some possibly non-trivial conditions in order for the equation to admit any solutions at all. In the compact setting, integrating this equation over the orbifold shows that $F$ must have average value zero. In the asymptotically conical case, however, it is natural in addition to require $F$ to satisfy some decay condition (so that one can integrate the equation over both sides). In particular, we could ask that $F$ decay at least as fast as $\rho^\beta$ for some weight $\beta < 0$, where $\rho$ is a radius function defined on the underlying asymptotically conical orbifold. Let us write $C_\beta^\infty(\mathcal{X})$ for the space of smooth functions on $\mathcal{X}$ decaying at least as fast as $\rho^\beta$ (together with their derivatives). (See Definition \ref{defn:AC} and the discussion that follows for a precise definition of $C_\beta^\infty(\mathcal{X})$.) Then with this convention, whenever $F$ belongs to $C_\beta^\infty(\mathcal{X})$, we would expect to seek solutions $\varphi$ belonging to $C_{\beta + 2}^\infty(\mathcal{X})$ (heuristically because taking two derivatives amounts to increasing the decay rate by an order of exactly two). 

The existence of solutions to the complex Monge-Amp\'ere equation \eqref{eqn:MA} depends on its linearization, which, with the requirement that $F$ belong to $C_\beta^\infty(\mathcal{X})$, becomes an equation involving the elliptic differential operator given by the Laplacian
\begin{align}\label{eqn:lap}
\Delta : C_{\beta + 2}^\infty(\mathcal{X}) \to C_{\beta}^\infty(\mathcal{X})
\end{align}
corresponding to some K\"ahler metric. Now, in the compact setting, the Fredholm index of the Laplacian is always zero (because the kernel and cokernel both consist of the one-dimensional subspace of constant functions). However, in this asymoptically conical setting, the Fredholm index of the Laplacian can vary---and may not even exist!---depending on the decay rate $\beta$ of the space of functions on which we consider the Laplacian to be defined. In particular, the operator \eqref{eqn:lap} is Fredholm for almost all weights $\beta$, except a discrete set $\mathcal{P}$ of ``exceptional weights'' (which depends only on the cone to which $\mathcal{X}$ is asymptotic).  Moreover, for an interval of weights not containing an exceptional one, the corresponding family of operators has the same Fredholm index. Concisely, the Fredholm index of \eqref{eqn:lap} may be considered as a monotonic function defined for almost all weights $\beta$, with only jump discontinuities at the exceptional weights (where the Fredholm index is not defined). 

The previous paragraph suggests that the existence and uniqueness of solutions to \eqref{eqn:MA} depend on the decay rate $\beta$, and this is indeed true.  In particular, there is an interval of weights where the Laplacian \eqref{eqn:lap} has Fredholm index zero, and for such a weight, it is known that the Monge-Amp\'ere equation \eqref{eqn:MA} can be solved uniquely. When the Fredholm index of the Laplacian is positive, we can also solve \eqref{eqn:MA}, but, due to the existence of a non-trivial kernel, we can no longer guarantee uniqueness. Finally, when the index is the first negative value, we cannot solve \eqref{eqn:MA} with the given space of functions, and must instead consider the space of functions $\mathbb{R}\rho^{2 - 2n} \oplus C_{\beta + 2}^\infty(\mathcal{X})$ to obtain a solution. Concisely, the results are summarized in the following. 

\begin{theorem}\label{thm:main}
Let $(\mathcal{X}, g)$ be a K\"ahler orbifold of complex dimension $n$ which is asymptotically conical of order $\lambda_{g}$ for some weight $\lambda_g<0$ (see Definition \ref{defn:AC} (i)). Let $\mathcal{P}$ denote the exceptional weights of the Laplacian (see Theorem \ref{thm:fredholm}), and let $\rho$ be a radius function on $\mathcal{X}$ (see Definition \ref{defn:AC} (ii)). Given a smooth function $F \in C_\beta^\infty(\mathcal{X})$, consider the complex Monge-Amp\'ere equation \eqref{eqn:MA} for a smooth function $\varphi$ on $\mathcal{X}$. 
\begin{enumerate}
\item[(i)] If $\beta \in (\max\{-4n, \beta_1^-, \lambda_g - 2n\}, -2n)$ where $\beta_1^-$ is the exceptional weight corresponding smallest nonzero eigenvalue of the Laplacian of $\Sigma$ (see \eqref{eqn:beta}),  then there is a unique solution $\varphi \in \mathbb{R} \rho^{2 - 2n} \oplus C^\infty_{\beta + 2}(\mathcal{X})$. 
\item[(ii)] If $\beta \in (-2n, -2)$, then there is a unique solution $\varphi \in C^\infty_{\beta + 2}(\mathcal{X})$. 
\item[(iii)] If $\beta \in (-2,0)$ and $\beta + 2 \notin \mathcal{P}$, then there is a solution $\varphi \in C^\infty_{\beta + 2}(\mathcal{X})$. 
\end{enumerate}
\end{theorem}

This result should be compared with the main existence result (Theorem 2.1) of \cite{ch13}, which is almost precisely the same statement in the setting of asymptotically conical \emph{manifolds} and which improves upon previous results \cite{goto, joyce, tianyau, bando, coevering2010, coeveringarxiv}, partly by unifying such results into a single framework. Because our notion of asymptotically conical orbifold requires all of the orbifold points to be contained in some compact set, proofs from the manifold setting (hence from \cite{ch13}) can go through mutatis mutandis. A small difference in our statement of Theorem \ref{thm:main} is the precision offered in the statement of the interval for case (i), which depends on a study of the existence of solutions to the Dirichlet problem for the Laplacian on asymptotically conical orbifolds (see Section \ref{sec:cases}), which should be compared to the case of asymptotically locally Euclidean (ALE) manifolds as in \cite{joyce}. Another difference from \cite{ch13} is a detailed outline of a proof of case (ii) using a continuity method involving estimates on weighted H\"older spaces, which is a streamlined modification of the approach of \cite{joyce} (which deals with the ALE case). 

Now let us suppose that the canonical bundle of the asymptotically conical K\"ahler orbifold $\mathcal{X}$ is trivial. Then since the zero form represents the first Chern class of $\mathcal{X}$, one would expect to be able to find metrics with vanishing Ricci curvature from Theorem \ref{thm:main}. However, in general, one would not expect every K\"ahler class to admit such a metric, instead, only those satisfying some sort of decay condition. Certainly compactly supported K\"ahler classes should admit Ricci-flat metrics (with the idea being that these are the classes which decay the fastest), but due to a version of the $\ddb$-lemma that holds outside of a compact set (see Lemma \ref{lem:ddbar}), one might also expect to find a Ricci-flat metric in K\"ahler classes that are almost compactly supported of some decay rate, in the sense that the class can be represented by a form which differs from an exact form by a form that decays with some weight (see Definition \ref{defn:mu}). More precisely, we have the following theorem concerning the existence of Calabi-Yau metrics (which is the orbifold analogue of  \cite[Theorem 2.4]{ch13}.)

\begin{theorem}\label{thm:calabiyau}
Let $\mathcal{X}$ be a complex orbifold of complex dimension $n > 2$ with trivial canonical bundle. Let $\Omega$ be a nowhere vanishing holomorphic volume form on $\mathcal{X}$. Let $\Sigma$ be Sasaki-Einstein with associated Calabi-Yau cone $(C, g_0, J_0, \Omega_0)$ and radius function $r$. Suppose that there is a constant $\lambda_\Omega < 0$, a compact subset $K \subset X$, and a diffeomorphism $\Phi : (1, \infty) \times \Sigma \to X \setminus K$ such that 
\[
\Phi^* \Omega - \Omega_0 = O(r^{\lambda_\Omega}).
\]
Let $\mathfrak{k} \in H^2(\mathcal{X})$ be an almost compactly supported K\"ahler class of rate $\lambda_{\mathfrak{k}} < 0$ (see Definition \ref{defn:mu}). Denote by $\lambda$ the maximum of $\lambda_{\mathfrak{k}}$ and $\lambda_\Omega$, and assume that $\lambda +2 \notin \mathcal{P} \cap (0,2)$. Then for all $c > 0$, there is an asymptotically conical Calabi-Yau metric $g_c$ on $\mathcal{X}$ whose associated K\"ahler form $\omega_c$ lies in $\mathfrak{k}$ and satisfies 
\[
{\Phi^* \omega_c - c \omega_0} = O(r^{\max\{-2n, \lambda\}}).
\]
Moreover, if $\lambda < -2n$, then there is an $\epsilon > 0$ such that 
\[
\Phi^* \omega_c - c \omega_0  = \textnormal{const} \sqrt{-1} \ddb r^{2 - 2n} + O(r^{-2n-1-\epsilon}).
\]
\end{theorem}

The contents of this paper are as follows. After a notion of asymptotically conical orbifold is introduced in Section \ref{sec:prelim}, the Fredholm index of the Laplacian of such orbifolds is studied in Section \ref{sec:fredholm}, thereby establishing existence and uniqueness results for the Dirichlet problem involving this Laplacian. The next section (Section \ref{sec:continuity}) outlines an approach to proving the main case (case (ii)) of Theorem \ref{thm:main} by using the continuity method of a one-parameter family of equations. Such a continuity method requires a priori estimates involving weighted H\"older norms on solutions to the one-parameter family of equations, and we reserve the next section (Section \ref{sec:apriori}) to derive these estimates, following the ALE approach of \cite{joyce} and using estimates for the Laplacian established in Section \ref{sec:fredholm}. Cases (i) and (iii) of Theorem \ref{thm:main} both follow from the existence of case (ii), but there are some additional arguments needed to show that the solution that is obtained decays with the desired rate, and these arguments are supplied in Section \ref{sec:cases}. Finally, the construction of Calabi-Yau (Ricci-flat) metrics is discussed in Section \ref{sec:calabiyau}, with specific attention given to the case of crepant partial resolutions of Calabi-Yau cones \ref{cor:crepant}, several examples of which are given in Section \ref{sec:examples}. 

One question that is not addressed in this paper is the uniqueness of the K\"ahler metric constructed by Theorem \ref{thm:main}. However, this question is discussed further detail in \cite{ch13} for the setting of manifolds. Indeed a main point of \cite{ch13} is that for uniqueness the decay condition $\mathcal{O}(r^{-n-\epsilon})$ needed in earlier work \cite{goto, joyce, coeveringarxiv} can be relaxed to $\mathcal{O}(r^{-\epsilon})$ using ideas about harmonic functions. We expect that such proofs carry over to the setting of orbifolds with the proper notational adjustments. 

For a more complete discussion of previous results in the setting of manifolds, we encourage the reader to consult \cite{ch13}. That paper also includes several families of examples that we neglect to discuss here (such as small resolutions associated to flag varieties and affine deformations of cones), but which we suspect possess similar generalizations to the orbifold setting. 
 
\section*{Acknowledgements}

The author would like to thank Hans-Joachim Hein, Daniel Litt, and Chiu-Chu Melissa Liu for helpful discussions and suggestions. The author is especially grateful to Chiu-Chu Melissa Liu for encouragement and support.  This material is based upon work supported by the National Science Foundation Graduate Research Fellowship Program under Grant No. DGE 16-44869. Any opinions, findings, and conclusions or recommendations expressed in this material are those of the author(s) and do not necessarily reflect the views of the National Science Foundation. 

\section{Preliminaries}\label{sec:prelim}

\begin{definition}\label{defn:cone}
Let $(\Sigma, g_\Sigma)$ be a compact Riemannian manifold. 

\begin{enumerate}
\item[(i)] The \emph{Riemannian cone} $C = C(\Sigma)$ over $\Sigma$ is defined to be the manifold $\mathbb{R}_+ \times \Sigma$ with the metric 
\[
g_0 = dr^2 + r^2 g_\Sigma
\]
where $r$ is a local coordinate on $\mathbb{R}_+ = (0,\infty)$. 
\item[(ii)] A tensor $T$ on the cone $C$ is said to decay with rate $\lambda \in \mathbb{R}$, written $T = O(r^\lambda)$, if $|\nabla_0^kT|_{g_0} = O(r^{\lambda - k})$ for each $k \in \mathbb{N}$, where $\nabla_0$ denotes the Levi-Civita connection corresponding to $g_0$. In particular, if $T$ is a tensor on $\Sigma$, we may regard $T$ as a tensor on the cone with decay rate $\lambda = -2$. 
\item[(iii)] We say that the cone $(C,g_0)$ is K\"ahler if $g_0$ is K\"ahler and $C$ is equipped with a choice of $g_0$-parallel complex structure $J_0$. In such a case, there is a K\"ahler form $\omega_0(U,V) = g_0(J_0 U, V)$ with local expression $\omega_0 = (i/2) \ddb r^2$. 
\end{enumerate}
\end{definition}

\begin{definition}\label{defn:AC}
Let $(\mathcal{X},g, J)$ be a complete K\"ahler orbifold with underlying space $X$, and let $(C,g_0, J_0)$ be a K\"ahler cone.  
\begin{enumerate}
\item[(i)] We say that $\mathcal{X}$ is asymptotically conical of rate $\lambda_g$ with tangent cone $C$ if there is a diffeomorphism $\Phi : C \setminus K \to X \setminus K'$, with $K,K'$ compact, such that $\Phi^*g - g_0 = O(r^{\lambda_g})$ and also $\Phi^*J - J_0 = O(r^{\lambda_g})$. (In particular, this means that all of the orbifold points of $\mathcal{X}$ are contained in $K'$.)
\item[(ii)] A radius function on an asymptotically conical $\mathcal{X}$ is a smooth function $\rho$ on $\mathcal{X}$ with codomain $[1,\infty)$ satisfying $\Phi^*\rho = r$ away from $K'$. 
\end{enumerate}
\end{definition}

A radius function $\rho$ on an asymptotically conical $\mathcal{X}$ allows us to define spaces of functions $C_\beta^k(\mathcal{X})$ in the following manner. For an integer $k \geqslant 0$ and a weight $\beta \in \mathbb{R}$, let $C_\beta^k(\mathcal{X})$ be the space of continuous functions $f$ with $k$ continuous derivatives such that the norm 
\[
\norm{f}_{C_\beta^k(\mathcal{X})} = \sum_{j=0}^k \sup_\mathcal{X}|\rho^{j-\beta} \nabla^j f|_g 
\]
is finite. Let $C^\infty_\beta(\mathcal{X})$ denote the intersection of all of the $C^k_\beta(\mathcal{X})$ for $k \geqslant 0$. 

We may also define weighted H\"older spaces $C_\beta^{k,\alpha}(\mathcal{X})$. The metric $g$ allows us to define the distance $d(x,y)$ between two points $x,y$ in the underlying space $X$ as the infimum of the lengths of all possible continuous \emph{admissible} curves connecting them (see \cite[Theorem 38]{borzellino} the notion of admissible curve and the notion of distance). In this way, we obtain the notion of a ball $B_r(x)$ of radius $r$ centered about $x$. The Levi-Civita connection for $g$ also allows us to say when a path is a geodesic, and hence we can say that a subset $Y$ of the underlying space is strongly convex if any two points are joined by a unique minimal geodesic entirely contained within $Y$. The convexity radius $r(x)$ at $x$ is defined to be the largest possible radius $R$ such that $B_r(x)$ is strongly convex for all $0 < r < R$. The convexity radius $r(g)$ of $g$ is the infimum over all $r(x)$. Taken over the compact set $K'$, the infimum will be positive by compactness, and over the rest of the orbifold, the infimum will be positive by the asymptotically conical assumption (and the compactness of $\Sigma$).  For a tensor $T$ on $\mathcal{X}$, we may then define the seminorm 
\[
[T]_{C_\beta^{0,\alpha}(\mathcal{X})} = \sup_{\substack{x \neq y \in X \\ d(x,y) < r(g)}} \left[\min(\rho(x), \rho(y))^{-\beta} \frac{|T(x) - T(y)|}{d(x,y)^\alpha}\right]
\]
where the distance $|T(x) - T(y)|$ is defined via parallel transport along the minimal geodesic from $x$ to $y$. Then define the weighted H\"older space $C^{k,\alpha}_\beta(\mathcal{X})$ to be the space of functions $f \in C^k_\beta(\mathcal{X})$ for which the norm 
\[
\norm{f}_{C_\beta^{k,\alpha}(\mathcal{X})} = \norm{f}_{C_\beta^k(\mathcal{X})} + [\nabla^kf]_{C^{0,\alpha}_{\beta - k - \alpha}}
\]
is finite.

For a pair of weights $\beta' < \beta$ and a pair of constants $\alpha' > \alpha$, the inclusion 
\begin{align}\label{eqn:AA}
C_{\beta'}^{k,\alpha'}(\mathcal{X}) \hookrightarrow C_{\beta}^{k,\alpha}(\mathcal{X})
\end{align}
is continuous.  In fact, analogous to the Arzela-Ascoli Theorem for compact manifolds/orbfiolds, this inclusion is compact, as stated below. 

\begin{theorem}\label{thm:AA}
For a pair of weights $\beta' < \beta$ and a pair of constants $\alpha' > \alpha$, the inclusion \eqref{eqn:AA} is compact. 
\end{theorem}

This theorem is proved for weighted H\"older spaces on complete non-compact manifolds in \cite[Lemme 3]{chaljub}. The arguments presented there can be applied to the setting of orbifolds with only minor notational adjustments.

\section{Fredholm index of the Laplacian}\label{sec:fredholm}

For an asymptotically conical $(\mathcal{X}, g)$ of complex dimension $n$, the Laplacian $\Delta$ induced from $g$ defines a linear map 
\begin{align}\label{eqn:laplacian}
\Delta : C^{k+2,\alpha}_{\beta + 2}(\mathcal{X}) \to C^{k,\alpha}_{\beta}(\mathcal{X})
\end{align}
for \emph{any} weight $\beta$. In the setting of compact orbifolds, the operator $\Delta$ is elliptic, and there are corresponding a prior estimates on solutions to equations involving $\Delta$ in corresponding (unweighted) H\"older spaces \cite{faulk}. The same is true in the setting of conical orbifolds in weighted H\"older spaces:

\begin{theorem}[Weighted Schauder estimates on conical orbifolds]\label{thm:schauder}
Let $(\mathcal{X}, g)$ be an asymptotically conical Riemannian orbifold, and let  $k,\ell \in \mathbb{N}$ and $\alpha \in (0,1)$.  Then there is a constant $C$ such that for each $f \in C^{k+2,\alpha}_{\beta + 2}(\mathcal{X})$, we have
\[
\norm{f}_{C^{k+2,\alpha}_{\beta + 2}(\mathcal{X})} \leqslant C(\norm{\Delta f}_{C^{k,\alpha}_{\beta}(\mathcal{X})} + \norm{f}_{C^0_{\beta + 2}(\mathcal{X})}). 
\]
\end{theorem}

To prove this theorem, we require first the result for manifolds, which is discussed in greater detail, for example, in Marshall \cite{marshal}. 

\begin{theorem}[Weighted Schauder estimates on conical manifolds]\label{thm:schaudercone}
Let $(C(\Sigma), g_\Sigma)$ be a Riemannian cone, and let $\alpha \in (0,1)$, $k \in \mathbb{N}$, and $\beta \in \mathbb{R}$. Then there is a constant $C$ such that for each $f \in C^{k+2,\alpha}_{\beta + 2}(C(\Sigma))$, we have 
\[
\norm{f}_{C^{k+2,\alpha}_{\beta + 2}(C(\Sigma))} \leqslant C(\norm{\Delta f}_{C^{k,\alpha}_{\beta}(C(\Sigma))} + \norm{f}_{C^0_{\beta + 2}(C(\Sigma))}). 
\]
\end{theorem}

\medskip

\noindent \emph{Proof of Theorem \ref{thm:schauder}}. Cover the compact set $K'$ by finitely many orbifold charts $(\widetilde{U}_\gamma, \pi_\gamma, G_\gamma)$. We may select $G_\gamma$-invariant relatively compact subsets $\widetilde{U}'_\gamma \subset \widetilde{U}_\gamma$ satisfying $d(\widetilde{U}'_\gamma, \partial \widetilde{U}_\gamma) > 0$ and whose supports still cover $K'$. Because the collection of supports $\{U_\gamma'\}$ covers $K'$, the norm $\norm{f}_{C^{k+2, \alpha}_{\beta + 2}(\mathcal{X})}$ is equivalent to the norm defined by 
\[
\norm{f}_{C^{k+2,\alpha}_{\beta + 2}(X \setminus K')} + \sum_{\gamma} \norm{f}_{C^{k+2, \alpha}(U_\gamma')}. 
\]
(Notice the absence of the weights in the norms over the subsets $U_\gamma'$.)
The ordinary Schauder estimates in $\mathbb{R}^n$ imply that for each $\gamma$, there is a constant $C_\gamma$ such that 
\[
\norm{f}_{C^{k+2,\alpha}(\widetilde{U}_\gamma')} \leqslant C_\gamma(\norm{\Delta f}_{C^{k,\alpha}(\widetilde{U}_\gamma)} + \norm{f}_{C^0(\widetilde{U}_\gamma)}).
\]
Moreover, the weighted Schauder estimates on conical manifolds (Theorem \ref{thm:consch}) together with the fact that $(\mathcal{X}, g)$ is asymptotically conical imply there is a constant $C'$ such that
\[
\norm{f}_{C^{k+2,\alpha}_{\beta + 2}(X \setminus K')} \leqslant C'(\norm{\Delta f}_{C^{k,\alpha}_{\beta}(X \setminus K')} + \norm{f}_{C^0_{\beta + 2}(X \setminus K')}). 
\]
If $C$ denotes the maximum of $C'$ and the numbers $C_\gamma$, then the result follows. \hfill $\Box$

\medskip

In the compact case, the Arzela-Ascoli Theorem---together with (unweighted) Schauder estimates---implies  that $\Delta$ is Fredholm. However, this implication fails to be true in the asympotically conical case, because, for example, $\mathcal{X}$ is not compact and hence neither is the embedding dealt with by Arzela-Ascoli. Nevertheless, for almost all weights $\beta$, the map \eqref{eqn:laplacian} is still Fredholm, as stated below. 

\begin{theorem}\label{thm:fredholm}
For an asymptotically conical K\"ahler orbifold $\mathcal{X}$ of complex dimension $n$ with K\"ahler cone $C(\Sigma)$, define the set of exceptional weights 
\[
\mathcal{P} = \left\{- \frac{m-2}{2} \pm \sqrt{\frac{(m-2)^2}{4} + \mu} : \mu \geqslant 0 \: \text{is an eigenvalue of $\Delta_\Sigma$}\right\}.
\]
where $m  = 2n =  \dim_{\mathbb{R}}C(\Sigma)$. Then the operator \eqref{eqn:laplacian} is Fredholm if $\beta + 2 \notin \mathcal{P}$. 
\end{theorem}

\begin{remark}
In general, the set $\mathcal{P}$ is disjoint from the interval $(-2n+2, 0)$ and symmetric about the point $1-n = (2-m)/2$. 
\end{remark}

To prove this result, we require a result from \cite{lm85} which states that an elliptic operator on the \emph{full cylinder} is an isomorphism away from the exceptional weights. The full cylinder is the product $\text{Cyl}(\Sigma) = \mathbb{R} \times \Sigma$ where we use the coordinate $t$ on $\mathbb{R}$ and $t$ is related to $r$ by the rule $e^t = r$. If $\Delta$ is the Laplacian on the cone, then we may regard $\Delta$ as a differential operator on the cylinder as well, and the composition $P = e^{2t} \circ \Delta$ can be regarded as an elliptic differential operator 
\begin{align}\label{eqn:P}
P : C_{\beta + 2}^{k+2, \alpha}(\textnormal{Cyl}(\Sigma)) \to C_{\beta + 2}^{k,\alpha}(\textnormal{Cyl}(\Sigma))
\end{align}
which is translation invariant, where the weighted H\"older spaces on the cylinder are exactly those on the cone with the change of variables $r = e^t$.   Away from the exceptional weights, it is known that this operator is actually an isomorphism.

\begin{lemma}\cite{lm85, goto}\label{lem:coneisom}
The operator \eqref{eqn:P} is an isomorphism provided $\beta + 2 \notin \mathcal{P}$. 
\end{lemma}

\medskip

\noindent \emph{Proof of Theorem \ref{thm:fredholm}}. The argument is essentially the same as that in \cite[Section 2]{lm85}: one splits the estimates into those near the apex of the cone and those near infinity. Since we have weighted conical estimates on an orbifold by Theorem \ref{thm:schauder}, we apply these near the apex,  and then combine these estimates with estimates near infinity from \ref{lem:coneisom} to obtain the desired result. 

More precisely, let $R$ be a number so large that $K' \subset \rho^{-1}([1,R])$, and set $X_1 = \rho^{-1}([1,R])$. Let $\varphi_1$ be a smooth cutoff function compactly supported on $X_1$ such that $\varphi_1 \equiv 1$ on $K'$. Set $\varphi_2 = 1 - \varphi_1$.  The Schauder estimates (Theorem \ref{thm:schauder}) give that
\[
\norm{\varphi_1f}_{C^{k+2,\alpha}_{\beta + 2}(\mathcal{X})} \leqslant C(\norm{\Delta (\varphi_1f)}_{C^{k,\alpha}_{\beta}(\mathcal{X})} + \norm{\varphi_1f}_{C^0_{\beta + 2}(\mathcal{X})}). 
\]
The composition $P = \rho^2 \circ \Delta$ is a translation invariant differential operator of order $2$ which is elliptic. For $\beta$ satisfying $\beta + 2 \notin \mathcal{P}$, view $\varphi_2 f$ as a function on the full cylinder $\textnormal{Cyl}(\Sigma)$, and then apply Lemma \ref{lem:coneisom} to the composition $P = \rho^2 \circ \Delta$ to obtain an estimate of the form 
\[
\norm{\varphi_2 f}_{C_{\beta + 2}^{k+2, \alpha}(\mathcal{X})} \leqslant C \norm{\rho^2 \Delta (\varphi_2 f)}_{C^{k,\alpha}_{\beta+2}(\mathcal{X})} = \norm{\Delta (\varphi_2 f)}_{C^{k,\alpha}_{\beta}(\mathcal{X})}
\]
Combining these two inequalities, we find that for $\beta$ satisfying $\beta + 2 \notin \mathcal{P}$, we have 
\begin{align}\label{eqn:compact}
\norm{f}_{C_{\beta + 2}^{k+2, \alpha}(\mathcal{X})} &\leqslant C\left(\norm{\varphi_1 \Delta f}_{C^{k,\alpha}_\beta(\mathcal{X})} + \norm{[\varphi_1, \Delta] f}_{C^{k,\alpha}_\beta(\mathcal{X})}\right. \\
&\left.+ \norm{\varphi_2 \Delta f}_{C^{k,\alpha}_\beta(\mathcal{X})} + \norm{[\varphi_2, \Delta]f}_{C^{k,\alpha}_\beta(\mathcal{X})} + \norm{\varphi_1f}_{C^0_{\beta + 2}(\mathcal{X})}\right) \notag
\end{align}
where $[\varphi_i, \Delta] = \varphi_i\Delta - \Delta \varphi_i$. 

For $\beta$ satisfying $\beta + 2 \notin \mathcal{P}$, define two maps  
\[
A_1, A_2 : C_{\beta+2}^{k+2,\alpha}(\mathcal{X}) \longrightarrow C_{\beta}^{k,\alpha}(\mathcal{X}) \oplus C_{\beta}^{k,\alpha}(\mathcal{X}) \oplus C_{\beta}^{k,\alpha}(\mathcal{X}) \oplus C_{\beta+2}^0(\mathcal{X})
\]
by the assignments
\begin{align*}
A_1(f) &= (\Delta f, -[\varphi_1, \Delta]f, -[\varphi_2, \Delta] f, -\varphi_1 f) \\
A_2(f) &= (0, [\varphi_1, \Delta]f, [\varphi_2, \Delta] f , \varphi_1 f).
\end{align*}
Inequality \eqref{eqn:compact} asserts that $A_1$ has trivial kernel and closed image. If we knew that $A_2$ were a compact mapping, then we could apply Proposition 3.9.7 of \cite{narasimhan} to conclude that $\Delta = A_1 + A_2$ has finite dimensional kernel and closed image. Therefore, it suffices to prove that $A_2$ is compact, which we do presently. 

We give only an argument to show that $[\varphi_1, \Delta]$ is a compact mapping 
\[
[\varphi_1, \Delta] : C_{\beta+2}^{k+2,\alpha}(\mathcal{X}) \to C_{\beta}^{k,\alpha}(\mathcal{X}),
\]
claiming that arguments for $[\varphi_2, \Delta]$ and $\varphi_1$ are similar. Direct computation shows that 
\[
[\Delta, \varphi_1]f = (\Delta \varphi_1) f - \langle \nabla \varphi_1, \nabla f \rangle.
\]
Because $\varphi_1$ is supported only on $X_1$, we conclude that $[\Delta, \varphi_1]f$ is supported only on $X_1$ as well. It follows that if $f_j$ is a sequence bounded in $C_{\beta+2}^{k+2, \alpha}(\mathcal{X})$, then $[\Delta, \varphi_1]f_j$ is a sequence bounded in $C^{k+1, \alpha}(X_1)$. The Arzela-Ascoli theorem applied to the compact $X_1$ ensures that there is a subsequence of $[\Delta, \varphi_1]f_j$ which converges in $C^{k,\alpha}(X_1)$ and hence also in $C_{\beta}^{k,\alpha}(\mathcal{X})$. 

Finally the arguments in \cite[Section 2]{lm85} shows that for $\beta + 2 \notin \mathcal{P}$, the map \eqref{eqn:laplacian} has finite co-dimensional range as well, with only minor suitable adaptations to the orbifold setting.  \hfill $\Box$

\medskip

For a weight $\beta < -2$, the kernel of \eqref{eqn:laplacian} is trivial, so that the Fredholm index of \eqref{eqn:laplacian} is nonpositive for non-exceptional weights $\beta < -2$. 

\begin{lemma}\label{lem:inj}
For $\beta < -2$ satisfying $\beta + 2 \notin \mathcal{P}$, the map \eqref{eqn:laplacian} is injective. 
\end{lemma}

\begin{proof}
Let $f \in C_{\beta + 2}^{k+2,\alpha}(\mathcal{X})$ be such that $\Delta f = 0$. For $R > 1$, let $T_R$ denote the compact subset of $X$ given by $T_R = \rho^{-1}([1,R])$. The restriction of $f$ to $T_R$ is harmonic and so achieves its maximum on the boundary, which we denote by $S_R$. In particular, the function $f$ belongs to $C^0_{\beta+2}(\mathcal{X})$ so that there is a constant $C$ such that 
\[
\max_{T_R}f = \sup_{S_R}f  < CR^{\beta+2}.
\] 
Taking $R \to \infty$ implies that $f$ must be identically zero. 
\end{proof}

\begin{lemma}\label{lem:lapsurj}
For $\beta > -2n$ satisfying $\beta + 2 \notin \mathcal{P}$, the map \eqref{eqn:laplacian} is surjective. 
\end{lemma}

\begin{proof}
The formal $L^2$-adjoint $\Delta^*$ is a mapping satisfying 
\[
\int_{\mathcal{X}} (\Delta f) h  \; dV_g  = \int_{\mathcal{X}} f (\Delta^*h) \; dV_g
\]
for compactly supported smooth functions $f,h$ on $\mathcal{X}$. Because the volume of $\mathcal{X}$ behaves as $O(\rho^{2n})$, we see that this identity extends to functions $f \in C_{\beta + 2}^\infty(\mathcal{X})$ and $h \in C_{-\beta -2n - \epsilon}^\infty(\mathcal{X})$ for any $\epsilon > 0$. Two integrations by parts show that we may take $\Delta^* = \Delta$. We also find that the image of $\Delta$ is contained in the subspace perpendicular to the kernel of  
\begin{align}\label{eqn:adjointepsilon}
\Delta : C_{-\beta - 2n-\epsilon}^{k+2, \alpha}(\mathcal{X}) \to C_{-\beta - 2n -2 - \epsilon}^{k,\alpha}(\mathcal{X}) 
\end{align}
for each $\epsilon > 0$. It therefore follows that the image of \eqref{eqn:laplacian} is contained in the subspace perpedicular to the kernel of 
\begin{align}\label{eqn:adjoint}
\Delta : C_{-\beta - 2n}^{k+2, \alpha}(\mathcal{X}) \to C_{-\beta - 2n -2 }^{k,\alpha}(\mathcal{X}).
\end{align} 
The results of \cite[Theorem 6.10]{marshal} extended to the setting of orbifolds imply that the image of \eqref{eqn:laplacian} is actually equal to the subspace perpendicular to the kernel of \eqref{eqn:adjoint}. For  $\beta > -2n$ satisfying $\beta + 2 \notin \mathcal{P}$, the previous lemma implies that the kernel of \eqref{eqn:adjoint} is trivial, so that \eqref{eqn:laplacian} is surjective. 
\end{proof}

\begin{corollary}\label{cor:isom}
For $\beta$ satisfying $-2n < \beta < -2$, the map \eqref{eqn:laplacian} is an isomorphism. 
\end{corollary}

\section{Case (ii) of Theorem \ref{thm:main}}\label{sec:continuity}

In this section, we solve equation \eqref{eqn:MA} in the case (ii) (where $-2n < \beta < -2$) by the method of continuity. (In particular, throughout the section, we will assume that $-2n < \beta < -2$.) For a parameter $t \in [0,1]$, consider the one-parameter family of equations 
\begin{align}\label{eqn:*_t}
\begin{cases}
(\omega + \sqrt{-1} \partial \bar{\partial} \varphi)^n = e^{tF} \omega^n \\
\omega + \sqrt{-1} \partial \bar{\partial} \varphi > 0
\end{cases} \tag{$*_t$}.
\end{align}
The equation $(*_0)$ admits the trivial solution $\varphi = 0$. The desired equation we want to solve is equation $(*_1)$. To solve this equation, it suffices to prove the following proposition. 

\begin{proposition}\label{prop:suff}  
\item[(i)] If \eqref{eqn:*_t} admits a smooth solution belonging to $C_{\beta+2}^\infty(\mathcal{X})$ for some $t < 1$, then for all sufficiently small $\epsilon > 0$, the equation $(*_{t + \epsilon})$ admits a smooth solution belonging to $C_{\beta+2}^\infty(\mathcal{X})$ as well. 
\item[(ii)] There is a constant $C > 0$ depending only on $\mathcal{X}, \omega, F,$ and $\alpha$ such that if $\varphi \in C^\infty_{\gamma}(\mathcal{X})$ satisfies \eqref{eqn:*_t} for some $t \in [0,1]$ and some weight $\gamma$ satisfying $\beta + 2 \leqslant \gamma < 0$, then 
\begin{itemize}
\item $\norm{\varphi}_{C^{3,\alpha}_{\beta+2}(\mathcal{X})} \leqslant C$ and
\item $(g_{j\bar{k}} + \partial_j \partial_{\bar{k}}\varphi) > C^{-1} (g_{j\bar{k}})$, where $g_{j\bar{k}}$ are the components of $\omega$ in local coordinates of any chart and the inequality means that the difference of matrices is positive definite. 
\end{itemize}
\end{proposition}

Indeed, Proposition \ref{prop:suff} is sufficient because we can obtain a solution to $(*_1)$ belonging to $C_{\beta+2}^\infty(\mathcal{X})$ using the following lemma. 

\begin{lemma}
Assume Proposition \ref{prop:suff}. Then if $s$ is a number in $(0,1]$ such that we can solve \eqref{eqn:*_t} for all $t < s$, then we can solve $(*_s)$ for a smooth function $\varphi$ belonging to $C_{\beta+2}^\infty(\mathcal{X})$. 
\end{lemma}

\begin{proof}
Let $t_i \in (0,1)$ be a sequence of numbers approaching $s$ from below. By assumption this gives rise to a sequence of smooth functions $\varphi_i$ satisfying 
\[
(\omega + \sqrt{-1} \ddb \varphi_i)^n = e^{t_iF}\omega^n.
\]
Proposition \ref{prop:suff}(i) and Theorem \ref{thm:AA} imply that after passing to a subsequence, we may assume that the $\varphi_i$ converge in $C^{3,\alpha'}_{\beta' + 2}$ to a function $\varphi$ for some $\alpha' > \alpha$ and $\beta' < \beta$. This convergence is strong enough that $\varphi$ satisfies the limiting equation 
\[
(\omega + \sqrt{-1}\ddb \varphi)^n = e^{sF} \omega^n.
\]
Proposition \ref{prop:suff}(ii) gives that the forms $\omega + \sqrt{-1}\ddb \varphi_i$ are bounded below by a fixed positive form $C^{-1}\omega$, and hence $\omega + \sqrt{-1} \ddb \varphi$ is a positive form. If we knew that $\varphi$ were smooth, then we could apply Proposition \ref{prop:suff}(i) to obtain that $\varphi$ belongs to $C^{\infty}_{\beta+2}(\mathcal{X})$. It thus remains to show that $\varphi$ is smooth. 

We show $\varphi$ is smooth by a standard local bootstrapping argument. In local coordinates, we find that $\varphi$ satisfies 
\[
\log \det(g_{j\bar{k}} + \partial_j \partial_{\bar{k}}\varphi) - \log \det (g_{j \bar{k}}) - sF = 0.
\]
Differentiating the equation with respect to the variable $z^\ell$ we have 
\[
(g_\varphi)^{j\bar{k}} \partial_j \partial_{\bar{k}} (\partial_\ell \varphi) =  s \partial_\ell F +  \partial_{\ell} \log \det (g_{j\bar{k}}) - (g_\varphi)^{j\bar{k}} \partial_\ell g_{j \bar{k}}
\]
where $(g_\varphi)^{j\bar{k}}$ is the inverse of the matrix $(g_\varphi)_{j\bar{k}} = g_{j\bar{k}} + \partial_j \partial_{\bar{k}} \varphi$. We think of this equation as a linear elliptic second-order equation $\Delta_\varphi(\partial_{\ell}\varphi) = h$ for the function $\partial_{\ell}\varphi \in C^{2,\alpha'}(\mathcal{X})$. Because the function $h$ belongs to $C^{1,\alpha'}$, we conclude from ordinary (unweighted) Schauder estimates that $\partial_\ell\varphi$ belongs to $C^{3,\alpha'}$. Because $\ell$ was arbitrary, it follows that $\varphi$ belongs to $C^{4,\alpha'}$. Repeating this argument we obtain that $\varphi \in C^{5,\alpha'}$ and by induction, that $\varphi$ is actually smooth. 
\end{proof}

Let us now prove the first part of Proposition \ref{prop:suff}. 

\medskip

\noindent \emph{Proof of Proposition \ref{prop:suff} (i)}. Let $B_1$ denote the Banach manifold consisting of those $\varphi \in C^{3,\alpha}_{\beta+2}(\mathcal{X})$ such that $\omega + \sqrt{-1} \partial \bar{\partial}\varphi$ is a positive form. Let $B_2$ denote the Banach space $B_2 =  C^{1,\alpha}_{\beta}(\mathcal{X})$. Define a mapping 
\begin{align*}
G : B_1 \times [0,1] &\longrightarrow  B_2 \\
(\varphi, s) &\longmapsto \log \frac{(\omega + \sqrt{-1} \partial \bar{\partial}\varphi)^n}{\omega^n} - sF.
\end{align*}
By assumption, we are given a smooth function $\varphi_t$ belonging to $C_{\beta+2}^\infty(\mathcal{X})$ such that $G(\varphi_t, t) = 0$ and $\omega + \sqrt{-1} \partial \bar{\partial} \varphi$ is a K\"ahler form. The partial derivative of $G$ in the direction of $\varphi$ at the point $(\varphi_t, t)$ is given by 
\[
DG_{(\varphi_t, t)}(\psi, 0) = \frac{n \sqrt{-1} \partial \bar{\partial} \psi \wedge \omega_t^{n-1}}{\omega_t^n} = \Delta_t \psi,
\]
where $\omega_t = \omega + \sqrt{-1} \ddb \varphi_t$ and $\Delta_t$ denotes the Laplacian with respect to $\omega_t$. Corollary \ref{cor:isom} gives that $\Delta_t$ is an isomorphism 
\[
\Delta_t : C_{\beta+ 2}^{3,\alpha}(\mathcal{X}) \to C^{1,\alpha}_{\beta}(\mathcal{X}).
\] 
The implicit function theorem implies that for $s$ sufficiently close to $t$, there are functions $\varphi_s$ in $C_{\beta+2}^{3,\alpha}(\mathcal{X})$ satisfying $(G(\varphi_s), s) = 0$. Because $\omega + \sqrt{-1}\ddb \varphi_t$ is a positive form, for $s$ close enough to $t$, we can ensure that each $\omega + \sqrt{-1}\ddb\varphi_s$ is a positive form as well. Moreover, bootstrapping arguments similar to those described earlier show that $\varphi_s$ is actually smooth. \hfill $\Box$

\medskip

It remains to prove Proposition \ref{prop:suff} (ii). We do this in the next section. 

\section{A priori estimates}\label{sec:apriori}

This section is devoted to proving Proposition \ref{prop:suff} (ii). In particular, we are still assuming that $-2n < \beta < -2$. While the proof is analogous to the compact setting (c.f. \cite{faulk}), there are a few main differences:

\begin{enumerate}
\item[(i)] Stokes' Theorem cannot be applied directly since our orbifold is not compact.
\item[(ii)] An a priori $L^2$-bound is replaced by an a priori $L^{p_0}$-bound for some $p_0 > $ (see Lemma \ref{lem:p_0}). 
\item[(iii)] Our bootstrapping arguments (which use weighted Schauder estimates) require a weighted $C^0_{\beta+2}$-estimate on solutions $\varphi$. 
\item[(iv)] Finally, the desired result is a \emph{weighted} $C^{k,\alpha}_{\beta+2}$-estimate, which we show follows from an unweighted $C^{k,\alpha}$-estimate (see Proposition \ref{prop:pass}) as is obtained in the compact setting. 
\end{enumerate}

Our methods in this section follow closely those of Joyce in \cite{joyce} with modifications motivated by a streamlined approach in the compact setting presented in \cite{faulk}, which was influenced strongly by the presentation of \cite{szekelyhidi14}.
 
Throughout, let us fix a weight $\gamma$ satisfying $\beta + 2 \leqslant \gamma$. We will be assuming that our solution $\varphi$ belongs to the weighted H\"older space $C_{\gamma}^\infty(\mathcal{X})$. Our goal is to obtain a priori estimates on $\varphi$. 
 
\subsection{A $C^0$-estimate}

\begin{lemma}\label{lem:est1}
For $p > (2-2n)/\gamma$, any smooth solution $\varphi \in C_\gamma^\infty(\mathcal{X})$ to \eqref{eqn:*_t} satisfies 
\[
\int_{\mathcal{X}} \left|\nabla |\varphi|^{p/2}\right|^{2}_g dV_g \leqslant \frac{np^2}{4(p-1)} \int_{\mathcal{X}} (1 - e^{F_t})\varphi |\varphi|^{p-2} dV_g.
\]
\end{lemma}

\begin{proof}
For sufficiently large $R$, if $T_R = \{x \in X : \rho(x) \leqslant R\}$, then by Stokes' Theorem, we have 
\begin{align*}
&\int_{T_R} d\left[\varphi |\varphi|^{p-2} d^c \varphi \wedge(\omega^{n-1} + \omega^{n-2} \wedge  \omega_{\varphi} + \cdots + \omega_\varphi^{n-1}) \right] \\
&= \int_{\partial T_R} \varphi |\varphi|^{p-2} d^c \varphi \wedge (\omega^{n-1} + \omega^{n-2} \wedge \omega_{\varphi} + \cdots + \omega_\varphi^{n-1}).
\end{align*}
Since $\varphi \in C^{\infty}_{\gamma}(\mathcal{X})$, on the boundary $\partial S_R$, we have that $\varphi = O(R^{\gamma}), d^c\varphi = O(R^{\gamma-1})$, and $\omega, \omega_\varphi = O(1)$. We also have that $\text{vol}(\partial S_R) = O(R^{2n-1})$. It follows that the right-hand side of the equality is $O(R^{p\gamma + 2n -2})$, where, by assumption on $p$, we have $p\gamma+ 2n -2 < 0$. Taking the limit as $R \to \infty$ gives that 
\[
\int_\mathcal{X} d\left[\varphi |\varphi|^{p-2} d^c \varphi \wedge(\omega^{n-1} + \omega^{n-2} \wedge \omega_{\varphi} + \cdots + \omega_\varphi^{n-1}) \right] = 0.
\]
Expanding the integrand gives the equation 
\begin{align*}
\int_{\mathcal{X}} \varphi|\varphi|^{p-2} (1- e^{F_t}) \omega^n = (p-1) \int_{\mathcal{X}} |\varphi|^{p-2} d\varphi \wedge d^c \varphi \wedge (\omega^{n-1} + \cdots + \omega_\varphi^{n-1}).
\end{align*}
Each term on the right is positive so that 
\[
\int_{\mathcal{X}} \varphi|\varphi|^{p-2} (1- e^{F_t}) \omega^n \geqslant (p-1) \int_{\mathcal{X}} |\varphi|^{p-2} d\varphi \wedge d^c \varphi \wedge \omega^{n-1}.
\]
But the right-hand side is 
\[
(p-1) \int_{\mathcal{X}} |\varphi|^{p-2} d\varphi \wedge d^c \varphi \wedge \omega^{n-1} = \frac{4(p-1)}{np^2} \int_{\mathcal{X}} \left|\nabla |\varphi|^{p/2}\right|^{2}_g \omega^n ,
\]
as we desire. 
\end{proof}

The Sobolev inequality is a statement that holds on both compact and non-compact manifolds and can be extended to asymoptotically conical orbifolds by a type of local argument which uses the estimate on each chart. In particular, if $U_i$ is a finite covering of $\Sigma$ by charts, then we can form a finite covering of $\mathcal{X}$  by taking a finite covering by charts of the compact $K'$ together with the finite collection of charts $U_i \times (0,\infty)$ on $C(\Sigma)$. The Sobolev inequality holds on each of these charts, so using a partition of unity argument we obtain a global Sobolev inequality, stated precisely below. (See also Corollary 1.3 of \cite{hein:sobolev}, which contains a more general treatment of Sobolev inequalities on non-compact manifolds.)

\begin{lemma}[Sobolev inequality]\label{lem:sobolev}
For $n \geqslant 2$, let $\tau = \frac{n}{n-1}$. There is a constant $C > 0$ depending on $(\mathcal{X},g)$ such that if $\varphi$ belongs to $L_1^2(\mathcal{X})$, then 
\[
\norm{\varphi}_{L^{2\tau}} \leqslant C \norm{\nabla \varphi}_{L^2}.
\] 
\end{lemma}

With the previous two results, one can obtain a uniform $L^{p_0}$-estimate: 

\begin{lemma}[An $L^{p_0}$-estimate]\label{lem:p_0}
There are constants $C > 0$ and $p_0$ larger than $\max\{(2-2n)/\gamma, 3n/2\}$ such that any solution $\varphi \in C^\infty_{\gamma}(\mathcal{X})$ to \eqref{eqn:*_t}  satisfies 
\[
\norm{\varphi}_{L^{p_0}} \leqslant C.
\]
\end{lemma}

\begin{proof}
Choose $p$ satisfying $p > 1$ and $p > (2-2n)/\gamma$. Let $q$ and $r$ be the numbers $q = np/(p + n - 1)$ and $r = \tau p/(p-1)$ where $\tau = n/(n-1)$,  and note that $q$ and $r$ satisfy $1/q + 1/r = 1$. Using Lemma \ref{lem:sobolev}, we have an estimate of the form 
\[
\norm{\varphi}_{L^{\tau p}}^p \leqslant C \norm{\nabla |\varphi|^{p/2}}^2_{L^2}.
\] 
Then we apply the result of the Lemma \ref{lem:est1} and H\"older's inequality to obtain that 
\begin{align*}
\norm{\varphi}_{L^{\tau p}}^p  &\leqslant \frac{Cnp^2}{4(p-1)} \int_{\mathcal{X}} (1 - e^{tF})\varphi |\varphi|^{p-2} dV_g \\
&\leqslant  \frac{Cnp^2}{4(p-1)} \norm{1-e^{tF}}_{L^q} \norm{|\varphi|^{p-1}}_{L^r}.
\end{align*}
But we note from the definition of $r= \tau p/(p-1)$ that 
\[
\norm{|\varphi|^{p-1}}_{L^r} = \norm{\varphi}_{L^{r(p-1)}}^{p-1} = \norm{\varphi}^{p-1}_{L^{\tau p}}.
\]
We conclude that for $C$ sufficiently large, we have
\[
\norm{\varphi}_{L^{\tau p}} \leqslant  Cp  \norm{1-e^{tF}}_{L^q}.
\]
The condition that $p > (2-2n)/\gamma$ implies that $q\beta < -2n$, so that $\norm{1 - e^{tF}}_{L^q}$ exists, and can be bounded by a constant depending on $\mathcal{X}, \omega,$ and $F$. Choosing $p_0 = \tau p$ completes the proof of the claim. 
\end{proof}

\begin{lemma}\label{lem:induc}
Setting $\tau = n/(n-1)$ and with $p_0$ as in Lemma \ref{lem:p_0}, for each integer $k \geqslant 0$, let $p_k = \tau^k p_0$. Then there are constants $A, B >0 $ such that any solution $\varphi \in C^\infty_{\gamma}(\mathcal{X})$ to \eqref{eqn:*_t}  satisfies 
\[
\norm{\varphi}_{L^{p_k}} \leqslant A (Bp_k)^{-n/p_k}. 
\]
\end{lemma}

\begin{proof}
The sequence of norms $\norm{1-e^{tF}}_{L^{p_k}}$ converges (to $\norm{1 - e^{tF}}_{C^0}$) and is hence bounded, meaning there is a constant $D$ depending only on $\mathcal{X}, \omega, F$ such that $\norm{1-e^{tF}}_{L^{p_k}} \leqslant D$ for each $k$.  If $C$ denotes the constant from Lemma \ref{lem:sobolev}, let $B > 1$ be a constant satisfying 
\[
\sqrt[3]{B} \geqslant CDn2^n \tau^{n-1}.
\]
Then let $A > 1$ denote a constant satisfying 
\[
A \geqslant (Bp_0)^{n/p_0} \norm{\varphi}_{L^{p_0}}. 
\]
With these choices for $A$ and $B$, we prove the claim by induction on the letter $k$. For $k = 0$, the claim is true by the definitions of $A$ and $B$. 

Now suppose the result has been proved for all integers less than or equal to $k$, and we prove the result for $k + 1$. If $r = p_k/(p_k - 1)$, then $1/p_k + 1/r = 1$. Using Lemma \ref{lem:sobolev}, we have an estimate of the form 
\[
\norm{\varphi}_{L^{p_{k+1}}}^{p_k} = \norm{\varphi}_{L^{\tau p_k}}^{p_k} \leqslant C \norm{\nabla |\varphi|^{p_k/2}}^2_{L^2}.
\]
We then apply Lemma \ref{lem:est1} and H\"older's inequality to find that 
\begin{align*}
\norm{\varphi}_{L^{p_{k+1}}}^{p_k} &\leqslant \frac{Cnp_k^2}{4(p_k-1)} \norm{1 - e^{tF}}_{L^{p_k}} \norm{|\varphi|^{p_k -1}}_{L^r} \\
&\leqslant  \frac{CDnp_k^2}{4(p_k-1)} \norm{|\varphi|^{p_k -1}}_{L^r}
\end{align*} 
by the definition of $D$.
Since $p_k > 1$, we have $p_k^2/(4(p_k-1)) \leqslant p_k$ so that 
\[
\norm{\varphi}_{L^{p_{k+1}}}^{p_k} \leqslant CDnp_k \norm{|\varphi|^{p_k-1}}_{L^r}. 
\]
But 
\[
\norm{|\varphi|^{p_k-1}}_{L^r} = \norm{\varphi}^{p_k-1}_{L^{r(p_k-1)}} = \norm{\varphi}^{p_k-1}_{L^{p_k}}
\]
implies that  
\[
\norm{\varphi}_{L^{p_{k+1}}}^{p_k} \leqslant CDnp_k \norm{\varphi}^{p_k-1}_{L^{p_k}}.
\]
We apply the inductive hypothesis to the right-hand side to find that 
\[
\norm{\varphi}_{L^{p_{k+1}}}^{p_k}\leqslant  A^{p_k-1} CDn p^{n/p_k}  B^{n/p_k-1}(Bp_k)^{1-n}.
\]
The quantity $A$ is larger than $1$, the inequality $p_k^{1/p_k} < 2$ is valid for any positive number $p_k >1$, and we are assuming that $p_0 > 3n/2$ (so that $B^{n/p_k-1} < B^{-1/3}$), so that we may obtain 
\[
\norm{\varphi}_{L^{p_{k+1}}}^{p_k} \leqslant  A^{p_k} CDn 2^{n}  B^{-1/3} (B p_k)^{1-n}.
\]
By our assumption on $B$, we have that $CDn2^nB^{-1/3} \leqslant \tau^{1-n}$, and we conclude that 
\[
\norm{\varphi}_{L^{p_{k+1}}}^{p_k} \leqslant A^{p_k} (B \tau p_k)^{1-n} = (A (B p_{k+1})^{-n/p_{k+1}})^{p_k} .
\]
This completes the inductive step and the proof. 
\end{proof}

\begin{proposition}[A $C^0$-estimate]
There is a constant $C$ such that any solution $\varphi \in C_{\gamma}^{\infty}(\mathcal{X})$ to \eqref{eqn:*_t} satisfies $\norm{\varphi}_{C^0(\mathcal{X})} \leqslant C$.
\end{proposition}

\begin{proof}
One uses the previous lemma to find that 
\begin{align*}
\norm{\varphi}_{C^0} &\leqslant \lim_{k \to \infty} A (B p_k)^{-n/p_k} = A
\end{align*}
as desired. 
\end{proof}

\subsection{A $C^3$-estimate}

Local calculations together with the $C^0$-estimate then imply the following $C^2$-estimate (see \cite{yau78, szekelyhidi14}).

\begin{proposition}[A $C^2$-estimate]\label{prop:C^2}
There is a constant $C$ depending on $\mathcal{X}, \omega, F$ such that a solution $\varphi \in C_\gamma^\infty(\mathcal{X})$ of \eqref{eqn:*_t} satisfies 
\[
C^{-1}(g_{j\bar{k}}) < (g_{j\bar{k}} + \partial_j \partial_{\bar{k}}\varphi) < C (g_{j\bar{k}})
\]
where $<$ means that the difference of matrices is positive definite and where $\omega$ has local expression $\omega = \sqrt{-1} g_{j\bar{k}} dz^j \wedge d\bar{z}^k$. 
\end{proposition}

Let $S$ denote the tensor given by the difference of Levi-Civita connections $S = \hat{\nabla} - \nabla$, where $\hat{\nabla}$ is the connection corresponding to $\omega_\varphi$ and $\nabla$ is the one corresponding to $\omega$. Note that $S$ depends on the second and third derivatives of $\varphi$. So if $|S|$ denotes the norm of $S$ with respect to the metric $\omega_\varphi$, the fact that the metric $g_{j\bar{k}}$ is uniformly equivalent to the metric $g_{j\bar{k}} + \partial_j \partial_{\bar{k}}\varphi$ implies that a bound on $|S|$ gives a $C^3$-bound on $\varphi$.

\begin{proposition}[A $C^3$-estimate]\label{prop:C^3}
There is a constant $C$ depending on $\mathcal{X}, \omega, F$ such that if $\varphi \in C_\gamma^\infty(\mathcal{X})$ is a solution to \eqref{eqn:*_t}, then $|S| \leqslant C$, where $|S|$ is the norm of $S$ computed with respect to the metric $\omega_\varphi$. 
\end{proposition}

Once a $C^3$-estimate is known, ordinary bootstrapping arguments together with Schauder estimates and the $C^0$-estimate then imply the following.

\begin{corollary}[A $C^{k,\alpha}$-estimate]
Let $k$ be a nonnegative integer and $\alpha \in (0,1)$. There is a constant $C$ depending on $\mathcal{X}, \omega, F$ such that if $\varphi \in C_\gamma^\infty(\mathcal{X})$ is a solution to \eqref{eqn:*_t}, then $\norm{\varphi}_{C^{k,\alpha}(\mathcal{X})} \leqslant C$.
\end{corollary}

\subsection{A weighted $C^0$-estimate}

We now prove a weighted $C^0$-estimate. Our goal is specifically to prove a $C^0_{\beta + 2}$-estimate. However, to do so, we must first prove a $C^0_{\gamma}$-estimate for a weight $\gamma$ satisfying $\beta + 2 < \gamma$ \emph{and} $-1 < \gamma$. A stronger estimate for the weight $\beta + 2$ will eventually follow (see Proposition \ref{prop:C^0_beta+2}).  

To prove the $C^0_\gamma$-estimate, we proceed in a way analogous to the unweighted $C^0$-estimate presented above, with minor adaptations to deal with the weight $\gamma$. 

The arguments in \cite[Proposition 8.6.7]{joyce} show directly (in the setting of asymptotically locally Euclidean manifolds) that the following is true. 

\begin{lemma}
There is a constant $C > 0$ such that for $p > (2-2n)/\gamma $ and $q \geqslant 0$ satisfying $p\gamma + q < 2 - 2n$, any solution $\varphi \in C_\gamma^\infty(\mathcal{X})$ to \eqref{eqn:*_t} satisfies 
\begin{align*}
\norm{\nabla (|\varphi|^{p/2} \rho^{q/2})}_{L^2}^2 &\leqslant \frac{np^2}{4(p-1)} \int_{\mathcal{X}} (1 - e^{tF}) \varphi |\varphi|^{p-2} \rho^q dV \\
&\;\; + C \frac{q(p+q)}{4(p-1)} \int_{\mathcal{X}} |\varphi|^p \rho^{q-2} dV. 
\end{align*}
\end{lemma}

In addition, the arguments of Proposition 8.6.8 of \cite{joyce} prove also the following weighted analogue of the Sobolev inequality. It is convenient to introduce the weighted Sobolev norm 
\[
\norm{f}_{L_{k,\beta}^q} = \left(\sum_{j=0}^k \int_{\mathcal{X}} |\rho^{j-\beta} \nabla^j f|^q \rho^{-2n} \: dV \right)^{1/q}.
\]

\begin{lemma}
Let $\gamma$ be a weight satisfying $\beta + 2 < \gamma$ and $-1 < \gamma$. There is a constant $C > 0$ such that if $p \geqslant 2$ satisfies $p  \geqslant (2 - 2n)/\gamma$ then any solution $\varphi \in C_\gamma^\infty(\mathcal{X})$ of \eqref{eqn:*_t} satisfies  
\[
\norm{\varphi}^p_{L^{\tau p}_{0,\gamma}} \leqslant C p (\norm{\varphi}_{L_{0,\gamma}^{p-1}}^{p-1} + \norm{\varphi}_{L_{0,\gamma}^p}^p).
\]
\end{lemma}

We may now obtain a uniform weighted $L^1$-estimate. 

\begin{lemma}\label{lem:L^p2}
Let $\gamma$ be a weight satisfying $\beta + 2 < \gamma$ and $-1 <  \gamma$. There is a constant $C > 0$ such that if $\varphi \in C^\infty_{\gamma}$ is a solution to \eqref{eqn:*_t} then $\norm{\varphi}_{L_{0,{\gamma}}^{1}} \leqslant C$.  
\end{lemma}

\begin{proof}
Let $p_0$ be chosen from Lemma \ref{lem:p_0}. Because $-1 < \gamma$, we may also ensure that $p_0$ satisfies 
\[
p_0  < \frac{-2n}{\gamma}.
\]
Define $r, s$ by the relations $1/p_0 + 1/r = 1$ and $s = -r(\gamma + 2n)$. Then by H\"older's inequality, we find that 
\[
\norm{\varphi}_{L_{0,\gamma}^1} = \int_{\mathcal{X}} |\varphi| \rho^{-\gamma - 2n} dV \leqslant \norm{\varphi}_{L^{p_0}} \left[\int_{\mathcal{X}} \rho^s dV \right]^{1/r}.
\]
The choice of $p_0$ satisfying $p_0\gamma > -2n $ ensures that $s < -2n$ so that the integral $\int_{\mathcal{X}} \rho^s dV$ exists. The result now follows from Lemma \ref{lem:p_0}.  
\end{proof}

By techniques similar to those used in Lemma \ref{lem:induc}, one can use the previous  lemmas to prove the following. 
 
\begin{lemma}
Let $\gamma$ be a weight satisfying $\beta + 2 < \gamma$ and $-1 < \gamma$. With $\tau = n/(n-1)$, for each integer $k \geqslant 0$, let $p_k = \tau^k$. Then there are constants $A,B > 0$ such that any solution $\varphi \in C_{\gamma}^\infty$ to \eqref{eqn:*_t} satisfies 
\[
\norm{\varphi}_{L_{0,\gamma}^{p_k}} \leqslant A(B p_k)^{-n/p_k}. 
\]
\end{lemma}

A $C^0_{\gamma}$-estimate now follows immediately. 

\begin{proposition}[A $C^0_{\gamma}$-estimate]\label{prop:C^0_gamma}
Let $\gamma$ be a weight satisfying $\beta + 2 < \gamma$ and $-1 <  \gamma$.  There is a constant $C$ such that any solution $\varphi \in C^\infty_{\gamma}(\mathcal{X})$ to \eqref{eqn:*_t} satisfies $\norm{\varphi}_{C^0_{\gamma}} \leqslant C$. 
\end{proposition}

\subsection{A weighted $C^3$-estimate}

The techniques of Theorem 8.6.11 from \cite{joyce} (which include a priori estimates of elliptic operators on domains in $\mathbb{C}^n$) can be used to show that a $C^0_\gamma$-estimate implies a $C^3_\gamma$-estimate. 

\begin{proposition}\label{prop:pass}
If $\gamma \geqslant \beta + 2$ is a weight such that we have an estimate of the form $\norm{\varphi}_{C_\gamma^0} \leqslant C$, then we also have an estimate of the form $\norm{\varphi}_{C_\gamma^3} \leqslant C$, and hence by weighted bootstrapping arguments involving the weighted Schauder estimates of Theorem \ref{thm:schauder}, we also have estimates of the form $\norm{\varphi}_{C_\gamma^{k,\alpha}} \leqslant C$ for each $k,\alpha$. 
\end{proposition}

Moreover, the next proposition asserts that, as soon as we have a $C^0_\gamma$-estimate, we may decrease the weight $\gamma$ to obtain an even stronger weighted estimate, so that we may continue until we obtain a $C^0_{\beta+2}$-estimate. We first require a lemma, the proof of which can be found in \cite{joyce} and involves choosing holomorphic coordinates and higher order estimates on $\varphi$. 

\begin{lemma}\label{lem:estlap}
For a solution $\varphi \in C_\gamma^\infty(\mathcal{X})$ of \eqref{eqn:*_t}, we have an estimate of the form 
\[
|\Delta \varphi + e^{tF} - 1| \leqslant C |\sqrt{-1} \ddb \varphi|^2 
\]
where $\Delta$ denotes the Laplacian and Levi-Civita connection of either $\omega$ or $\omega_\varphi$, since the corresponding metrics are equivalent by Proposition \ref{prop:C^2}. 
\end{lemma}

\begin{lemma}\label{lem:C^0_gamma}
Let $\gamma \geqslant \beta + 2$ be a weight such that we have an estimate of the form $\norm{\varphi}_{C_\gamma^0} \leqslant C$. If $\gamma' = \max\{2\gamma -2, \beta + 2\}$, then we also have an estimate of the form $\norm{\varphi}_{C_{\gamma'}^0} \leqslant C$. 
\end{lemma}

\begin{proof}
The idea is to use the previous lemma and the fact that the Laplacian \eqref{eqn:laplacian} is an isomorphism. In particular, since we are assuming we have an estimate of the form $\norm{\varphi}_{C_\gamma^0} \leqslant C$, Proposition \ref{prop:pass} shows that we actually have an estimate of the form $\norm{\varphi}_{C_\gamma^{5,\alpha}} \leqslant C$. From this, we conclude that $\sqrt{-1} \ddb \varphi \in C_{\gamma-2}^{3,\alpha}$, and thus that $|\sqrt{-1} \ddb \varphi|^2 \in C_{2(\gamma-2)}^{3,\alpha}$.  Lemma \ref{lem:estlap} establishes an estimate of the form 
\[
\Delta \varphi = O(\rho^{2\gamma -4}) + O(\rho^{\beta}).
\]
Corollary \ref{cor:isom} then gives the desired estimate. 
\end{proof}

\begin{proposition}[A $C^0_{\beta + 2}$-estimate]\label{prop:C^0_beta+2}
There is a constant $C$ such that any solution $\varphi \in C_{\gamma}^\infty$ to \eqref{eqn:*_t} satisfies $\norm{\varphi}_{C_{\beta +2}^0} \leqslant C$. 
\end{proposition}

\begin{proof}
Let $\gamma$ be a weight satisfying $\beta + 2 < \gamma$ and $-1 < \gamma$. Then Proposition \ref{prop:C^0_gamma} gives an estimate of the form $\norm{\varphi}_{C_\gamma^0} \leqslant C$. Define the sequence of weights $\gamma_0, \gamma_1, \ldots$ by the rule $\gamma_0 = \gamma$ and $\gamma_{i+1} = \max\{2\gamma_i - 2, \beta + 2\}$. Then for all sufficiently large $i$, we have $\gamma_i = \beta + 2$, and the previous lemma therefore gives an estimate of the form $\norm{\varphi}_{C_{\beta +2}^0} \leqslant C$.
\end{proof}

\subsection{Proof of Proposition \ref{prop:suff}(ii)} Proposition \ref{prop:C^0_beta+2} gives an estimate of the form $\norm{\varphi}_{C_{\beta +2}^0} \leqslant C$. Proposition \ref{prop:pass} then implies that we have an estimate of the form $\norm{\varphi}_{C_{\beta + 2}^{k,\alpha}} \leqslant C$ for each $k,\alpha$. Moreover, the metrics determined by $\omega$ and $\omega + \sqrt{-1} \ddb \varphi$ are equivalent by  Proposition \ref{prop:C^2}.

\section{Cases (i) and (iii) of Theorem \ref{thm:main}}\label{sec:cases}

It remains to discuss Theorem in the cases (i) and (iii), that is, if the right-hand side decays fast and slowly respectively. We require some preliminary results. 

\begin{lemma}\label{lem:laplaceexis}
Suppose $\beta$ satisfies $\beta < -n-1$ and $\beta + 2 \notin \mathcal{P}$. For any $ f \in C_\beta^\infty(\mathcal{X})$, there is a unique $u \in C^\infty(\mathcal{X}) \cap L_{1,1-n}^2(\mathcal{X})$ such that $\Delta u = f$, where $L_{1,1-n}^2(\mathcal{X})$ denotes the weighted Sobolev space given by the completion of the space of compactly supported smooth functions with respect to the weighted Sobolev norm 
\begin{align*}
\norm{v}_{L_{1,1-n}^2}^2 = \int_{\mathcal{X}} \left(|v|^2 \rho^{-2} + |\nabla v |^2 \right)\: dV.
\end{align*}
\end{lemma}

\begin{proof}
Define a functional 
\[
E(v) = \int_{\mathcal{X}} \left(\frac{1}{2} |\nabla v|^2 + f v\right) \: dV
\]
for all functions $v$ in $L_{1,1-n}^2(\mathcal{X})$. 

We first claim that there are constants $\delta, A > 0$ such that 
\begin{align}\label{eqn:coercive}
E(v) \geqslant \delta \norm{v}_{L_{1,1-n}^2}^2 - A.
\end{align}
Indeed, using H\"older's inequality, we can bound $E(v)$ from below by 
\begin{align*}
E(v) \geqslant \frac{1}{2} \norm{\nabla v}_{L^2}^2 - \norm{\rho^{-1}v}_{L^2} \norm{\rho f}_{L^2},
\end{align*}
where the $L^2$-norm of $\rho f$ is finite because $\beta < -n-1$. 
A geometric mean type of inequality implies that for each $\epsilon > 0$ we have 
\begin{align*}
E(v) \geqslant \frac{1}{2}\norm{\nabla v}_{L^2}^2 - \frac{\epsilon}{2} \norm{\rho^{-1}v}_{L^2}^2 - \frac{1}{2\epsilon} \norm{\rho f}_{L^2}^2.
\end{align*}
An orbifold version of Theorem 1.2(ii) from \cite{hein:sobolev} (with $\alpha = 1$ and $\beta = 2n$) gives a weighted Poincar\'e inequality on $\mathcal{X}$ of the form 
\begin{align}\label{eqn:Poincare}
\norm{\rho^{-1}v}_{L^2}^2 = \norm{v}_{L_{0,1-n}^2}^2 \leqslant C \norm{\nabla v}_{L^2}^2. 
\end{align}
It follows that for $\epsilon$ sufficiently small, we have 
\[
E(v) \geqslant \frac{1}{4}\norm{\nabla v}_{L^2}^2  - \frac{1}{2\epsilon} \norm{\rho f}_{L^2}^2. 
\]
We use \eqref{eqn:Poincare} again to find that 
\[
E(v) \geqslant \frac{1}{8}\norm{\nabla v}_{L^2}^2 + \frac{1}{8C} \norm{\rho^{-1}v}_{L^2}^2 - \frac{1}{2\epsilon} \norm{\rho f}_{L^2}^2,
\]
and hence, for $C$ large enough, we have 
\[
E(v) \geqslant \frac{1}{8C} \norm{v}_{L_{1,1-n}^2}^2 - \frac{1}{2\epsilon}\norm{\rho f}_{L^2}^2. 
\]

By the calculus of variations \cite{vainberg}, the functional $E$ has a unique critical point $u \in L_{1,1-n}^2$ which achieves an absolute minimum of $E$, and moreover $u$ is a weak solution to the equation $\Delta u = f$. It then follows from standard (local) elliptic regularity arguments that $u$ is actually smooth (since $f$ is). 
\end{proof}

\begin{lemma}(c.f. \cite[Proposition 8.3.4]{joyce})
For an asymptotically conical orbifold $(\mathcal{X}, g)$ with radius function $\rho$, we have $\Delta(\rho^{2-2n}) \in C_{\lambda_g - 2n}^\infty(\mathcal{X})$ and 
\[
\int_{\mathcal{X}} \Delta(\rho^{2-2n}) \: dV = (2n- 2) \textnormal{Vol}(\Sigma)
\]
where $\textnormal{Vol}(\Sigma)$ is the volume of the compact manifold $\Sigma$. 
\end{lemma}

\begin{proof}
For the statement about the decay rate of $\Delta(\rho^{2-2n})$, we know that $\Delta(r^{2-2n}) = 0$ on the cone $C(\Sigma)$. It follows from the definition of the radius function and the fact that $\mathcal{X}$ is asymptotically conical that $\Delta(\rho^{2-2n})$ belongs to $C_{\lambda_g - 2n}^\infty(\mathcal{X})$. 

Let $S_R$ be the subset of $X$ given by $S_R = \{x \in X : \rho(x) \leqslant R\}$. Then Stokes' Theorem gives that 
\[
\int_{S_R} \Delta(\rho^{2-2n})\: dV = \int_{\partial S_R} [\nabla(\rho^{2-2n}) \cdot \mathbf{n}] \: dV.
\]
For $R$ large enough, the quantity $\nabla(\rho^{2-2n}) \cdot \mathbf{n}$ is approximated by $(2n-2) R^{1-2n}$ and $\text{vol}(\partial S_R)$ is approximated by $R^{2n-1} \text{Vol}(\Sigma)$. Letting $R$ tend to $\infty$ gives the desired result. 
\end{proof}

Let $\mu_1$ be the smallest nonzero eigenvalue of $\Delta_{\Sigma}$, and let $\beta_1^{\pm}$ be the exceptional weights corresponding to this eigenvalue in the sense that 
\begin{align}\label{eqn:beta}
\beta_1^{\pm } = - \frac{2n-2}{2} \pm \sqrt{\frac{(2n-2)^2}{4} + \mu_1}. 
\end{align}

\begin{lemma}\label{lem:laplacedecay}
Suppose $\beta$ satisfies $\beta_1^- < \beta < -2n$, and let $f$ belong to $C_{\beta}^\infty(\mathcal{X})$. 
\begin{enumerate}
\item[(a)] If $\int_{\mathcal{X}} f \: dV = 0$, then the unique solution $u \in C^\infty(\mathcal{X}) \cap L_{1,1-n}^2(\mathcal{X})$ to $\Delta u = f$ belongs to the space $C_{\beta + 2}^\infty(\mathcal{X})$. 
\item[(b)] If $\int_{\mathcal{X}} f \: dV \ne 0$ and $\beta$ satisfies $\lambda_g - 2n < \beta$, then the unique solution $u \in C^\infty(\mathcal{X}) \cap L_{1,1-n}^2(\mathcal{X})$ to $\Delta u = f$ can be written as $u = A \rho^{2-2n} + v$ for a unique number $A$ and a unique function $v \in C_{\beta + 2}^\infty(\mathcal{X})$. Moreover, the number $A$ can be computed explicitly as 
\[
A = \frac{1}{(2n-2)\textnormal{Vol}(\Sigma)} \int_{\mathcal{X}} f \: dV.
\]
\end{enumerate}
\end{lemma}

\begin{proof}
For part (a), the proof of Lemma \ref{lem:lapsurj} states that the range of 
\begin{align}\label{eqn:lapsmooth}
\Delta : C_{\beta+2}^{k+2,\alpha}(\mathcal{X}) \to C_{\beta}^{k,\alpha}(\mathcal{X}) 
\end{align}
is the orthogonal complement of the kernel of
\begin{align}\label{eqn:lapadjoint}
\Delta : C_{-\beta - 2n }^{k+2,\alpha}(\mathcal{X}) \to C_{-\beta - 2n - 2}^{k,\alpha}(\mathcal{X}).
\end{align}
Our assumption on $\beta$ guarantees that $-\beta - 2n$ belongs to the interval $(0, \beta_1^+)$. In this interval, the kernel of \eqref{eqn:lapadjoint} is the one-dimensional subspace spanned by the constant $1$ function.  It follows that the range of \eqref{eqn:lapsmooth} is the subspace $W$ of all $f \in C_{\beta}^{k,\alpha}(\mathcal{X})$ satisfying $\int_{\mathcal{X}} f \: dV = 0$. The restriction 
\[
\Delta : C_{\beta+ 2}^{k,\alpha}(\mathcal{X}) \to W
\]
is an isomorphism, and hence there is an estimate of the form 
\[
\norm{u}_{C_{\beta+2}^{k+2,\alpha}} \leqslant C \norm{f}_{C_{\beta}^{k,\alpha}} \hspace{5mm} \text{for $f = \Delta u$ satisfying $\int_{\mathcal{X}} f \: dV = 0$.}
\]
Claim (a) now follows. 

For part (b), the integral $\int_{\mathcal{X}} \Delta \rho^{2-2n} \: dV$ is finite and equal to $(2n-2) \text{Vol}(\Sigma)$ by the previous lemma. Because $\beta$ satisfies $\beta < -2n$, the integral $\int_{\mathcal{X}} f \: dV$ is also finite. Let $A$ denote the constant 
\[
A =  \frac{\int_{\mathcal{X}} f \: dV}{\int_{\mathcal{X}} \Delta \rho^{2-2n} \: dV} =  \frac{1}{(2n-2)\textnormal{vol}(\Sigma)} \int_{\mathcal{X}} f \: dV. 
\]
Since $\beta$ satisfies $\lambda_g - 2n < \beta$, the function $f - \Delta(A \rho^{2-2n})$ belongs to $C_{\beta}^\infty(\mathcal{X})$ and has zero integral over $\mathcal{X}$. By part (i), there is a unique $v \in C_{\beta + 2}^\infty(\mathcal{X})$ such that $\Delta v = f - \Delta(A \rho^{2-2n})$. Upon rearranging, we find that the proof is complete. 
\end{proof}

\subsection{Case (i)}

Let us now discuss the case (i) of Theorem \ref{thm:main}. In this case, we are assuming that $\beta$ satisfies $ \max\{-4n, \beta_1^-, \lambda_g - 2n\} < \beta < -2n$. The idea is that there is an inclusion $C_{\beta}^\infty(\mathcal{X}) \hookrightarrow C_{\beta'}^\infty(\mathcal{X})$ for $\beta'$ satisfying $\beta < -2n < \beta'$ so that we can use the existence from case (ii), noting, however, that the solution we obtain may not belong to the desired space of functions. Nevertheless, because the solution satisfies a Monge-Amp\'ere equation, we will be able to use Lemma \ref{lem:laplacedecay}(b) to conclude that the solution belongs to the space we want. 

More precisely, for a small number $\epsilon > 0$, there is an inclusion $C_{\beta}^\infty(\mathcal{X}) \hookrightarrow C_{-2n + \epsilon}^\infty(\mathcal{X})$. It therefore follows from case (ii), that there is a unique solution $\varphi$ to \eqref{eqn:MA} satisfying $\varphi \in C_{2 - 2n + \epsilon}^\infty(\mathcal{X})$. By expanding the Monge-Amp\'ere equation \eqref{eqn:MA} and using the relation  
\[
n \sqrt{-1} \ddb \varphi \wedge \omega^{n-1} = (\Delta \varphi) \omega^n,
\]
we find that 
\[
(\Delta \varphi) \omega^n = (1 - e^F) \omega^n  + \sum_{j=2}^n {{n}\choose{j}} (\sqrt{-1} \ddb \varphi)^j \wedge \omega^{n-j}.
\]
The term $(1 - e^F) \omega^n $ belongs to $O(\rho^\beta)$ by assumption on $F$. All of the terms in the summation belong to $O(\rho^{-4n + 2\epsilon})$ because $j \geqslant 2$. It follows that $\Delta \varphi$ belongs to $O(\rho^\beta)$, and hence by 
Lemma \ref{lem:laplacedecay}(b) we find that $\varphi  \in \mathbb{R} \rho^{2-2n} \oplus C_{\beta + 2}^\infty(\mathcal{X})$ as desired.

\subsection{Case (iii)}

We finish by discussing the case (iii). In this case, we are assuming that $\beta$ satisfies $-2 < \beta < 0$ and $\beta + 2 \notin \mathcal{P}$. The idea is to reduce again to the case (ii), using the following lemma. 

\begin{lemma}
Suppose that $\beta$ satisfies $-2 < \beta < 0$ and $\beta + 2 \notin \mathcal{P}$. If $F$ belongs to $C_\beta^\infty(\mathcal{X})$, then there is a function $\varphi_1 \in C_{\beta + 2}^\infty(\mathcal{X})$ satisfying 
\begin{align}\label{eqn:MA2}
\begin{cases}
(\omega + \sqrt{-1} \ddb \varphi_1)^n = e^{F - F_1} \\
\omega + \sqrt{-1} \ddb \varphi_1 > 0
\end{cases}
\end{align}
for some $F_1 \in C_{2\beta}^\infty(\mathcal{X})$. 
\end{lemma}

\begin{proof}
Identify $X \setminus K'$ with $(1,\infty) \times \Sigma$. Let  $\eta : \mathbb{R}_+ \to \mathbb{R}_+$ be a smooth function satisfying 
\[
\eta(t) = \begin{cases}
0 & r \leqslant 1 \\
1 & r \geqslant 2
\end{cases}.
\]
For $R \geqslant 1$, let $\eta_R$ be the composition $\eta_R(r) = \eta(r/R)$. 
Since $\Delta : C_{\beta + 2}^\infty(\mathcal{X}) \to C_{\beta}^\infty(\mathcal{X})$ is surjective by Lemma \ref{lem:lapsurj}, there is a function $\hat{\varphi}_1 \in C_{\beta + 2}^\infty(\mathcal{X})$ such that $\Delta \hat{\varphi}_1  =  F$ on $\mathcal{X}$. We claim that $\varphi_1 = \eta_R \hat{\varphi}_1$ has the desired properties for $R$ sufficiently large. 

We first claim that the form $\omega + \sqrt{-1} \ddb \varphi_1$ is a positive form. Indeed, we compute that 
\begin{align*}
\sqrt{-1} \ddb \varphi_1  &= \eta_R \sqrt{-1} \ddb \hat{\varphi}_1 + \sqrt{-1} \frac{\eta'}{R}(\partial \hat{\varphi}_1 \wedge \bar{\partial}r + \partial r \wedge  \bar{\partial} \hat{\varphi}_1) \\
&\;\;\;\;  + \sqrt{-1} \hat{\varphi}_1 \left(\frac{\eta'}{R} \ddb r + \frac{\eta''}{R^2} \partial r \wedge \bar{\partial} r \right). 
\end{align*}
Because $\hat{\varphi}_1 \in C_{\beta + 2}^\infty(\mathcal{X})$ and $\eta_R$ is supported only for $r > R$, we find that the length of the first term $\eta_R \sqrt{-1} \ddb \hat{\varphi}_1$ is $O(R^\beta)$. The derivatives of $\eta$ are only supported for $r \in [R, 2R]$ so that the lengths of the other terms are also $O(R^\beta)$. We conclude that $\sup |\sqrt{-1} \ddb \varphi_1| \to 0$ as $R \to \infty$, and so we can ensure that the form $\omega + \sqrt{-1} \ddb \varphi_1$ is positive for $R$ large enough. 

Now for the complex Monge-Amp\'ere equation, we note that for $r > 2R$, we have $\varphi_1 = \hat{\varphi}_1$ and hence 
\begin{align*}
(\omega + \sqrt{-1} \ddb \varphi_1)^n &= \left(1 + \frac{1}{2} \Delta \hat{\varphi}_1 \right) \omega^n + \sum_{j \geqslant 2} {{n}\choose{j}} (\sqrt{-1} \ddb \hat{\varphi}_1)^j \wedge \omega^{n-j} \\
&= (1 + F + O(r^{2\beta}))\omega^n.
\end{align*}
If we set $F_1 = F - \log(1 + F + O(r^{2\beta}))$, then we have (by the Taylor series for $\log$) that $F_1 = O(r^{2\beta})$. The result follows. 
\end{proof}

Indeed, we may now finish the proof of case (iii) in the following manner. By the previous lemma, we obtain a $\varphi_1 \in C_{\beta + 2}^\infty(\mathcal{X})$ and an $F_1 \in C_{2\beta}^\infty(\mathcal{X})$ satisfying \eqref{eqn:MA2}. In particular, the form $\omega_1 = \omega + \sqrt{-1} \ddb \varphi_1$ is a K\"ahler form. If it happens that $2\beta < -2$, then we can use case (ii) to obtain a solution $\varphi_2 \in C_{2\beta + 2}^\infty(\mathcal{X})$ to the equation $(\omega_1 + \sqrt{-1} \ddb \varphi_2)^n = e^{F_1}\omega_1^n$, and setting $\varphi = \varphi_1 + \varphi_2$, we find that $\varphi$ belongs to $C_{\beta}^\infty(\mathcal{X})$ and that $\varphi$ solves \eqref{eqn:MA}. If $2\beta$ is not smaller than $-2$, then we can use the previous argument to find $\varphi_2 \in C_{2\beta + 2}^\infty(\mathcal{X})$ such that $(\omega_1 + \sqrt{-1} \ddb \varphi_2)^n = e^{F_1 - F_2}\omega_1^n$ for some $F_2 \in C_{4\beta}^\infty(\mathcal{X})$. If $4\beta < -2$, then we can use the previous argument to solve the desired equation \eqref{eqn:MA}. If not, then we can proceed iteratively to solve \eqref{eqn:MA} in a finite number of steps.

\section{Calabi-Yau metrics}\label{sec:calabiyau}

This section is devoted to proving Theorem \ref{thm:calabiyau}, which states the existence of Ricci-flat metrics in certain K\"ahler classes which decay rapidly in the following precise sense. 

\begin{definition}\label{defn:mu}
Let $\mathcal{X}$ be a compact orbifold and $K \subset X$ a compact subset of the underlying space, and let $C = \mathbb{R}_+ \times \Sigma$ be a cone with cone metric $g_0$. Suppose that there is a diffeomorphism $\Phi : (1, \infty) \times \Sigma \to X \setminus K$. A class $\mathfrak{k}$ in $H^2(\mathcal{X})$ is called almost compactly supported of rate $\lambda_{\mathfrak{k}} < 0$ if the class can be represented by a K\"ahler form $\omega$ on $\mathcal{X}$ such that there is a compact set $K' \supset K$ and a real smooth $1$-form $\eta$ on $X \setminus K'$ such that the difference $\omega - d\eta$ decays with rate $\lambda_{\mathfrak{k}}$. 
\end{definition}

Our proof of the theorem follows very closely the proof presented in \cite[Theorem 2.4]{ch13} and relies upon the following lemmas.  The first lemma can be found in \cite{ch13}, and the proof given there still holds in the orbifold setting because the arguments are given outside of the compact subset $K'$ in which our orbifold is isomorphic to a cone. 

\begin{lemma}
With the hypotheses of Theorem \ref{thm:calabiyau}, we have $\Phi^*J - J_0 = O(r^{\lambda_{\Omega}})$ (in the sense of Definition \ref{defn:cone}). 
\end{lemma}

The second lemma can also be found in \cite{ch13}, and the exact same proof involving cut-off functions extends to the orbifold case with almost no adjustments. 

\begin{lemma}\label{lem:pluri}
With the hypotheses of Theorem \ref{thm:calabiyau}, for each $\alpha > 0$, there is a smooth plurisubharmonic function $h_\alpha$ on $\mathcal{X}$ which is strictly plurisubharmonic and whose pullback to $(1, \infty) \times \Sigma$ agrees with $r^{2\alpha}$ outside of a compact subset $K_\alpha \subset X$. 
\end{lemma}

The final lemma is a version of the $\ddb$-lemma that holds outside of a compact subset. Again, this lemma follows from the manifold case simply because away from the orbifold points, our orbifold is isomorphic to a manifold, so that the results \cite[Proposition A.2(ii), Corollary A.3(ii)]{ch13} still hold. 

\begin{lemma}\label{lem:ddbar}
Let $\mathcal{X}$ be an AC K\"ahler orbifold with trivial canonical bundle. If $n = \dim_{\mathbb{C}}\mathcal{X} > 2$ and if $\alpha$ is an exact real $(1,1)$-form on $X \setminus K$ for some compact $K$ containing all of the orbifold points, then there is a compact $K' \supset K$ and a smooth function $u$ on $X \setminus K'$ such that $\alpha = \sqrt{-1}\ddb u$ on $X \setminus K'$. 
\end{lemma}

We are now in a position to prove Theorem \ref{thm:calabiyau}. The proof of Theorem 2.4 from \cite{ch13} applies almost directly, but we sketch the arguments here for completeness.

\medskip

\noindent \emph{Proof of Theorem \ref{thm:calabiyau}}. We identify $X \setminus K$ with $(1, \infty) \times \Sigma$ via $\Phi$, and we let ourselves work with increasingly large compact subsets $K$ if necessary. By assumption, there is a smooth $1$-form $\eta$ on $X \setminus K$ such that the difference $\omega - d\eta$ decays with rate $\lambda_{\mathfrak{k}}$. By Lemma \ref{lem:ddbar}, there is a smooth function $u$ on $X \setminus K$ such that $d\eta =- \sqrt{-1}\ddb u$. By Lemma \ref{lem:pluri}, for each $\alpha > 0$, there is a smooth plurisubharmonic function $h_\alpha$ on $\mathcal{X}$ which is strictly plurisubharmonic and whose pullback to $(1, \infty) \times \Sigma$ agrees with $r^{2\alpha}$ outside of some compact subset $K_\alpha \subset X$. 

Fix some $\alpha \in (0,1)$. Ensure that the compact set $K$ contains $K_\alpha$ and $K_1$. Let $R$ be a number so large that $K \subset \{r \leqslant R\}$. Fix a cutoff function $\psi$ on $\mathcal{X}$ satisfying 
\[
\psi(x) = \begin{cases}
0 & \rho(x) < 2R \\
1 & \rho(x) > 3R.
\end{cases}
\] 
For a constant $S > 2$, let $\psi_S$ denote the rescaled cutoff function satisfying by 
\[
\psi_S(x) = \begin{cases}
0 & \rho(x) < 2RS \\
1 & \rho(x) > 3RS.
\end{cases}
\]
For a constant $c > 0$ and a constant $C$, let $\hat{\omega}$ be the form 
\[
\hat{\omega} = \omega + \sqrt{-1} \ddb(\psi u) + C \sqrt{-1} \ddb((1-\psi_S) h_\alpha) + c \sqrt{-1} \ddb h_1.
\]
In \cite{ch13}, it is shown that for suitable choices of $S, c,$ and $C$, the form $\hat{\omega}$ is a K\"ahler form on $\mathcal{X}$ in such a way that $(\mathcal{X}, \hat{\omega})$ is asymptotically conical of rate $\lambda < 0$. The K\"ahler form $\hat{\omega}$ has global Ricci potential given by 
\[
\hat{f} = \log\left(\frac{i^{n^2} \Omega \wedge \bar{\Omega}}{(\hat{\omega}/c)^n}\right)
\]
belonging to the space $C_\lambda^\infty(\mathcal{X})$.  We would like to use Theorem \ref{thm:main} to solve the equation 
\[
(\hat{\omega} + \sqrt{-1}\ddb \hat{\varphi})^n = e^{\hat{f}}{\hat{\omega}}^n
\]
for $\hat{\varphi}$, and we would obtain a Ricci-flat metric. Let us consider two cases for $\lambda$: either $\lambda < -2n$ or $-2n < \lambda < 0$. 
\begin{enumerate}
\item[(i)] If $\lambda < -2n$, then by considering $\lambda' \geqslant \lambda$ in the interval $(\max\{-4n, \beta_1^-, \lambda_g - 2n\}, -2n)$ of case (i) of Theorem \ref{thm:main} and the inclusion $C_\lambda^\infty(\mathcal{X}) \hookrightarrow C_{\lambda'}^\infty(\mathcal{X})$, we may view $\hat{f}$ as having decay rate $\lambda'$, and therefore use case (i) of Theorem \ref{thm:main} to obtain a solution $\hat{\varphi} \in \mathbb{R} \rho^{2 - 2n} \oplus C^\infty_{\beta + 2}(\mathcal{X})$ whose corresponding K\"ahler form decays with rate $-2n = \max\{\lambda, -2n\}$. 
\item[(ii)] If $-2n < \lambda < 0$, then we may use either case (ii) or (iii) to find a solution $\hat{\varphi} \in C_{\lambda + 2}^\infty(\mathcal{X})$ whose corresponding K\"ahler form decays with rate $\lambda = \max\{\lambda, -2n\}$. 
\end{enumerate}

\begin{remark}
In \cite[Remark 2.10]{ch13}, it is shown using the Lichnerowicz-Obata Theorem that if $\text{Ric}(g_0) \geqslant 0$, then $\mathcal{P} \cap (0,2) = \mathcal{P} \cap [1,2)$ and moreover that the exceptional weights in the interval $(1,2)$ are associated with the growth rates of plurisubharmonic functions on the cone $C$. This remark justifies the slight difference in the statement of Theorem \ref{thm:calabiyau} from that of \cite[Theorem 2.4]{ch13}.
\end{remark}

\begin{corollary}\label{cor:crepant}
Let $(C(\Sigma), g_0, J_0, \Omega_0)$ be a Calabi-Yau cone of complex dimension $n > 2$, let $p : C \to \Sigma$ denote the radial projection, and let $V$ be the normal affine variety associated to $C$. Let $\pi : \mathcal{X} \to V$ be a crepant partial resolution by an orbifold $\mathcal{X}$, and let $\mathfrak{k} \in H^2(\mathcal{X})$ be a class that contains positive $(1,1)$-forms. Then for each $c > 0$, there is a complete Calabi-Yau metric $g_c$ on $\mathcal{X}$ such that $\omega_c \in \mathfrak{k}$ and 
\begin{align}\label{eqn:crepantconical}
\omega_c - \pi^*{c\omega_0} = O(r^{-2+\delta})
\end{align}
for sufficiently small $\delta$. If $\mathfrak{k} \in H^2_c(\mathcal{X})$, then we have 
\[
\omega_c - \pi^*(c\omega_0) = \textnormal{const} \sqrt{-1} \ddb r^{2-2n} + O(r^{-2n -1 - \epsilon})
\]
for some $\epsilon > 0$. 
\end{corollary}

\begin{proof}
We first claim that we have an exact sequence of the form 
\[
0 \to H^2_c(\mathcal{X}, \mathbb{R}) \to H^2(\mathcal{X}, \mathbb{R}) \to H^{1,1}_{\text{pr,b}}(\Sigma).
\]
Indeed, if $X_1 \subset \mathcal{X}$ denotes the suborbifold $X_1 = \{x \in X : \rho(x) \leqslant 1\}$, then we may view $\Sigma$ as the boundary of $X_1$. Considering the pair $(X_1, \Sigma)$, we have a long exact sequence in cohomology of the form 
\[
\cdots \to H^{k-1}(\Sigma, \mathbb{R}) \to  H^k(X_1, \Sigma, \mathbb{R}) \to H^k(X_1, \mathbb{R}) \to H^k(\Sigma, \mathbb{R}) \to \cdots
\]
Using the identifications $H^k(X_1, \mathbb{R}) \simeq H^k(\mathcal{X}, \mathbb{R})$ and $H^k(X_1, \Sigma, \mathbb{R}) \simeq H^k_c(\mathcal{X}, \mathbb{R})$, we obtain a long exact sequence, a portion of which is  
\[
\cdots \to H^{1}(\Sigma, \mathbb{R}) \to  H^2_c(\mathcal{X}, \mathbb{R}) \to H^2(\mathcal{X}, \mathbb{R}) \to H^2(\Sigma, \mathbb{R}) \to \cdots
\] 
In \cite{coevering2010}, it is shown that $H^1(\Sigma, \mathbb{R}) = 0$ (because, for example, we may choose a metric on $\Sigma$ with positive Ricci curvature). Moreover, the Bochner formula (see \cite[Lemma 5.3]{goto}) gives that $H^2(\Sigma)$ can be identified with$H^{1,1}_{\text{pr,b}}(\Sigma)$, the primitive basic $(1,1)$-cohomology group associated with the Sasaki structure on $\Sigma$. The claim now follows.

Let $\omega$ be a closed positive $(1,1)$-form in the class $\mathfrak{k}$.  From exact sequence of the previous paragraph, there is a compact subset $K \subset X$ such that away from $K$, we have  
\[
\omega = d\eta + p^*\xi
\] 
for some real $1$-form $\eta$ and some primitive basic $(1,1)$-form $\xi$ on $\Sigma$.  Noting that $p^*\xi = O(r^{-2})$ shows that we can take $\lambda_{\mathfrak{k}} = -2 + \delta$ (and the fact that $\Omega$ agrees with $\Omega_0$ outside of a compact set implies that $\lambda_{\Omega} = -\infty$). Theorem \ref{thm:calabiyau} now gives the result. 
\end{proof}

\begin{remark}
The arguments in \cite{goto} (see Proof of Theorem 5.1) can actually be used to give a stronger version of Corollary \ref{cor:crepant}, whereby the relation \eqref{eqn:crepantconical} is replaced by the relation 
\[
\omega_c  - \pi^*c \omega_0 = p^*\xi + O(r^{-4})
\]
where $\xi$ is the primitive basic harmonic $(1,1)$-form on $\Sigma$ that represents the restriction of $\kappa$ to $\Sigma$. If $\xi = 0$, or equivalently, if $\mathfrak{k} \in H^2_c(\mathcal{X})$, then we have 
\[
\omega_c - \pi^*(c\omega_0) = \textnormal{const} \sqrt{-1} \ddb r^{2-2n} + O(r^{-2n -1 - \epsilon})
\]
for some $\epsilon > 0$.  
\end{remark}

\begin{remark}\label{rem:surface}
Moreover, the same arguments and method of proof in \cite{goto} (see Proof of Theorem 5.1)  can be used to deal with the surface case ($n = 2$) of Corollary \ref{cor:crepant}. In particular, in this case, the K\"ahler class must belong to $H_c^2(\mathcal{X})$. 
\end{remark}

\section{Examples}\label{sec:examples}

We consider examples of  crepant partial resolutions of Calabi-Yau cones, to which one can apply the results of Corollary \ref{cor:crepant} (and Remark \ref{rem:surface} for the case $n=2$) to obtain Ricci-flat K\"ahler metrics. 

Our first example is that of the canonical bundle over projective space $\mathbb{CP}^{n-1}$, which is actually a manifold, and which is covered by the results in \cite{ch13}. However, we find it useful to review this particular example, as it contains a construction that will be repeated in further examples. 

\begin{example}\label{eg:projective}
Projective space $\mathbb{CP}^{n-1}$ equipped with the Fubini-Study metric is a K\"ahler-Einstein Fano manifold of dimension $n-1$ (with K\"ahler-Einstein constant $n$). The tautological line bundle $\mathcal{O}(-1)$ is a Hermitian-Einstein vector bundle over $\mathbb{CP}^{n-1}$ when equipped with the hermitian metric $h$ induced by viewing $\mathcal{O}(-1)$ as a subbundle of the trivial vector bundle of rank $n$. Let $t$ denote the smooth nonnegative function on the total space $L$ of $\mathcal{O}(-1)$ defined by 
\[
t(\eta) = h_x(\eta, \eta) = |\eta|^2
\]
for $\eta$ a vector in the fiber of $L_x$ over $x \in \mathbb{CP}^{n-1}$. Let $\Sigma \subset L$ denote the corresponding $S^1$-bundle given by $\Sigma = t^{-1}(1)$. Then $\Sigma$ may be identified with the sphere $S^{2n-1}$, viewed as an $S^1$-bundle over $\mathbb{CP}^{n-1}$ by considering the inclusion $S^{2n-1} \hookrightarrow \mathbb{C}^{n} \setminus 0$ followed by the projection onto $\mathbb{CP}^{n-1}$. The group $\mathbb{Z}_n$ of $n$th roots of unity acts freely on $\Sigma$ via the diagonal action of $\mathbb{Z}_n$ on $S^{2n-1} \subset \mathbb{C}^n \setminus 0$. The variety $V$ associated to the cone $C(\Sigma/\mathbb{Z}_n) = C(S^{2n-1}/\mathbb{Z}_n)$ may be identified with the quotient variety $\mathbb{C}^n/\mathbb{Z}_n$, which carries a global holomorphic volume form from that of $\mathbb{C}^n$. There is a crepant resolution
\[
\pi : K_{\mathbb{CP}^{n-1}} \to \mathbb{C}^n/\mathbb{Z}_n,
\]  
which contracts the zero section of $K_{\mathbb{CP}^{n-1}}$ to the singular point of $\mathbb{C}^n/\mathbb{Z}_n$. Calabi \cite{calabi} lifts the K\"ahler metric on $\mathbb{CP}^{n-1}$ to a Sasaki-Einstein metric on the $S^1$-bundle $\Sigma \simeq S^{2n-1}$, so that the cone $C(\Sigma/\mathbb{Z}_n) \simeq C(S^{2n-1}/\mathbb{Z}_n)$ is a Calabi-Yau cone. 

Corollary \ref{cor:crepant} now abstractly proves the existence of a one-parameter family of AC Calabi-Yau metrics on $K_{\mathbb{CP}^{n-1}}$ in each K\"ahler class $\mathfrak{k}$ that contains positive $(1,1)$-forms. In particular, by solving an ODE, Calabi \cite{calabi} explicitly constructs a family of Ricci-flat K\"ahler metrics on the total space of the canonical bundle $K_{\mathbb{CP}^{n-1}} = \mathcal{O}(-n)$, and the classes represented by these metrics are compactly supported. 
\end{example}

The next example we consider is a genearlization of the previous example in the sense that we consider the canonical bundle over \emph{any} K\"ahler-Einstein Fano manifold. Again this example is actually a manifold and is covered by the previous results from \cite{ch13}. 

\begin{example}
More generally, let $(M, g)$ be a K\"ahler-Einstein Fano manifold of dimension $n-1$ with K\"ahler-Einstein constant $k_0$. Let $L$ denote the total space of a maximal root of the canonical bundle (meaning that if $\iota$ is the largest integer that divides $K_M$ in $\text{Pic}(M)$, then $L^\iota = K_M$). The function $(\det g)^{-1}$ describes a hermitian metric on $K_M$ with Hermitian-Einstein constant $k_0$. The corresponding metric on $L$ described by $h = (\det g)^{-1/\iota}$ is also Hermitian-Einstein with constant $\ell = k_0/ \iota$. Let $t$ denote the smooth nonnegative function on the total space of $L$ determined by $h$, and let $\Sigma$ denote the corresponding $S^1$-bundle over $M$ given by $\Sigma = t^{-1}(1) \subset L$. The fiberwise action of $\mathbb{Z}_{\iota}$ on $\Sigma$ is free.  The total space of the canonical bundle $K_M$ is a smooth crepant resolution of the variety associated to the cone $C(\Sigma/\mathbb{Z}_{\iota})$, which enjoys a global holomorphic volume form. Again by lifting the metric on $M$, Calabi \cite{calabi} constructs a Sasaki-Einstein metric on $\Sigma$, and in this way, the cone $C(\Sigma/\mathbb{Z}_\iota)$ enjoys the structure of a Calabi-Yau cone. 

Corollary \ref{cor:crepant} now abstractly proves the existence of a one-parameter family of AC Calabi-Yau metrics on $K_M$ in each K\"ahler class $\mathfrak{k}$ that contains positive $(1,1)$-forms. Again, by solving an ODE, Calabi \cite{calabi} explicitly constructs a family of Ricci-flat K\"ahler metrics on the total space of the canonical bundle $K_M$, and the classes represented by these metrics are compactly supported. 
\end{example}

The next family of examples are orbifolds of the form $[L/\Gamma]$, where $L$ is the tautological line bundle over $\mathbb{CP}^{n-1}$ and $\Gamma$ is a finite subgroup of $SU(n)$ acting upon $L$. Such orbifolds are not just asymptotically conical, but actually asymptotically locally Euclidean, meaning that the cone can be taken to be the one associated with $\mathbb{C}^n/\Gamma$. 

\begin{example}
Let $\Gamma$ be a finite subgroup of $SU(n)$ acting freely on $S^{2n-1} \subset \mathbb{C}^n$. Upon identifying $S^{2n-1}$ with the $S^1$-bundle $\Sigma$ over $\mathbb{CP}^{n-1}$ of Example \ref{eg:projective}, we obtain a free action of $\Gamma$ on $\Sigma$, which we may extend to an action on the total space $L$ of $\mathcal{O}(-1)$. The global quotient orbifold $[L/\Gamma]$ is a crepant partial resolution of the variety $\mathbb{C}^n/\Gamma$ associated to the Calabi-Yau cone $C(\Sigma/\Gamma) = C(S^{2n-1}/\Gamma)$. Corollary \ref{cor:crepant} now ensures the existence of a one-parameter family of AC Calabi-Yau metrics on $[L/\Gamma]$ in each K\"ahler class $\mathfrak{k}$ that contains positive $(1,1)$-forms.
\end{example}

\begin{example}\label{eg:kronheimer}
Let us fix our attention to the variety $X = \mathbb{C}^2/\Gamma$ for a finite subgroup $\Gamma$ of $SU(2)$. The ADE classification gives a one-to-one correspondence between finite subgroups of $SU(2)$ and simply laced Dynkin diagrams of the form $A_n$ for $n \geqslant 1$, $D_n$ for $n \geqslant 4$, $E_6$, $E_7$ and $E_8$. Moreover, if $\pi : \tilde{X} \to X$ denotes the minimal resolution of $X = \mathbb{C}^2/\Gamma$, then the Dynkin diagram is the dual graph of the exceptional set of the resolution, which is a union of $\#\text{Vert}$ copies of $\mathbb{P}^1$, where $\#\text{Vert}$ is the number of vertices in the Dynkin diagram corresponding to $\Gamma$.  Using this correspondence, Kronheimer \cite{kronheimer} constructs ALE hyper-K\"ahler metrics on the minimal resolution $\tilde{X}$.  Remark \ref{rem:surface} implies in addition that any intermediate crepant partial resolution $\underline{X}$ factoring $\pi$
\[
\tilde{X} \to \underline{X} \to X
\] 
admits a $b_2(\underline{X})$-parameter family of AC Calabi-Yau metrics as well. In this case, each K\"ahler class is compactly supported and  
\begin{align*}
b_2(\underline{X}) = \dim H^2_c(\underline{X}).
\end{align*}
Moreover, the inequality 
\[
b_2(\underline{X}) \leqslant b_2(\tilde{X}) = \#\text{Vert}
\]
always holds. 
\end{example}

\begin{example}
One can consider higher rank bundles as well. Let $L$ denote the total space of two copies $\mathcal{O}(-1)^{\oplus 2}$ of the tautological line bundle over $\mathbb{CP}^{1}$, and let $h$ denote the hermitian metric obtained as the product of the metrics induced on each factor separately. If $t$ denotes the corresponding nonnegative smooth function on the total space $L$, then the subset $\Sigma = t^{-1}(1) \subset L$ is an $S^3$-bundle over $\mathbb{CP}^1$. Moreover, one can show that $\Sigma$ is isomorphic to the trivial bundle $S^3 \times \mathbb{CP}^1 \simeq S^3 \times S^2$. The variety $V$ associated to the cone $C(\Sigma) \simeq C(S^3 \times S^2)$ may be identified with the affine variety $V = \{z_1^2 + z_2^2 + z_3^2 + z_3^2 = 0\} \subset \mathbb{C}^4$ considered by \cite{candelas}. There is a crepant resolution 
\[
\pi: L \to V
\]
which contracts the zero section of $L$ to the singular point of $V$. There is a Sasaki-Einstein metric on $S^3 \times S^2$ so that $C(S^3 \times S^2)$ becomes a Calabi-Yau cone. 

Corollary \ref{cor:crepant} now abstractly proves the existence of a one-parameter family of AC Calabi-Yau metrics on $L$ in each K\"ahler class $\mathfrak{k}$ that contains positive $(1,1)$-forms. We note that such K\"ahler classes are not compactly supported because in fact, if $E$ denotes the zero section, which is isomorphic to $\mathbb{CP}^1$, then $H_c^2(L) \simeq H_{2n-2}(E) = 0$. Moreover, since $b_2(L) = 1$, there is at most a one-parameter family of such K\"ahler classes that contain positive $(1,1)$-forms.  Moreover, in \cite{candelas}, an explicit one-parameter family of AC K\"ahler metrics on $L$ is constructed. 
\end{example}

\begin{example}
More generally, if $\Gamma$ is a finite subgroup of $SU(2)$ acting freely on $S^{3} \subset \mathbb{C}^2$, then we obtain a corresponding action of $\Gamma$ on the sphere bundle $\Sigma$, and hence also on the total space $L$ of two copies $\mathcal{O}(-1)^{\oplus 2}$ of the tautological line bundle over $\mathbb{CP}^{1}$. The global quotient orbifold $[L/\Gamma]$ is a crepant partial resolution of the variety associated to the Calabi-Yau cone $C((S^3/\Gamma) \times S^2)$, so Corollary \ref{cor:crepant} abstractly proves the existence of a family of AC Calabi-Yau metrics on $[L/\Gamma]$ in each K\"ahler class that contains $(1,1)$-forms. Moreover, the orbifold $[L/\Gamma]$ may be resolved fully to obtain a smooth resolution $\tilde{Y}$, and in analogy with Example \ref{eg:kronheimer}, each intermediate partial resolution $\underline{Y}$ admits a $b_2(\underline{Y})$-parameter family of AC Calabi-Yau metrics. In this case, each $\underline{Y}$ is a $\mathbb{CP}^1$-fibration of a partial resolution $\underline{X}$ of the variety $X = \mathbb{C}^2/\Gamma$ from Example \ref{eg:kronheimer}, and the second Betti number of $\underline{Y}$ satisfies 
\[
b_2(\underline{Y}) = 1 + \dim H_c^2(\underline{Y}) = 1 + b_2(\underline{X}).
\]
\end{example}

\bibliographystyle{abbrv}
\bibliography{YauACbib}

\end{document}